\documentclass{article}
\usepackage{amsthm,amsmath,amssymb}
\usepackage[margin=.5in,bottom=1in]{geometry} 
\usepackage{url}
\usepackage{caption}
\usepackage{subcaption}
\usepackage{graphicx}
\usepackage{epstopdf}
\usepackage{color}
\usepackage{framed}
\usepackage{lscape}
\usepackage{rotating}
\usepackage{algorithm}
\usepackage[noend]{algpseudocode}
\usepackage{grffile}
\usepackage{multirow}
\usepackage{xcolor}
\usepackage{comment}
\usepackage{xr}

\newtheorem{thm}{Theorem}[section]

\newtheorem{lem}[thm]{Lemma}
\newtheorem{cor}[thm]{Corollary}
\newtheorem{prop}[thm]{Proposition}

\theoremstyle{definition}
\newtheorem{rmk}{Remark}

\newcommand{\D}{\mathrm{D}}

\newcommand{\tr}{\mathrm{tr}}

\newcommand{\EE}{\mathbb{E}}
\newcommand{\R}{\mathbb{R}}

\newcommand{\calG}{\mathcal{G}}
\newcommand{\calP}{\mathcal{P}}
\newcommand{\calU}{\mathcal{U}}

\newcommand{\SO}{\mathrm{SO}}

\newcommand{\la}{\langle}
\newcommand{\ra}{\rangle}

\newcommand{\what}{\widehat}
\newcommand{\wtilde}{\widetilde}

\newcommand{\ZZ}{\mathbb{Z}}
\newcommand{\RR}{\mathbb{R}}

\newcommand{\sign}{\mathop{\mathrm{sign}}}

\DeclareMathOperator*{\argmax}{\arg\!\max}

\begin{document}
\title{\textbf{On the stability of scale-space metrics}}
\author{\textbf{William Leeb} \\ \\
School of Mathematics \\
University of Minnesota, Twin Cities  \\
Minneapolis, MN}
\date{}
\maketitle

\abstract{
We study the stability of a classical family of
metrics defined over functions' Gaussian scale-space
representations,
focusing on the comparison of images (functions of two variables).
These metrics have precedents both in harmonic analysis,
specifically the theory of Besov spaces, and in classical
methods of image processing; special cases
are also known to be metrically equivalent to certain
Wasserstein distances.
We quantify these metrics' robustness to geometric
deformations,
and introduce rotationally-invariant
versions that are stable to
changes in angle when comparing
tomographic projections.
We also describe computationally efficient algorithms
for evaluating the metrics from finite samples,
and prove their robustness to additive noise.
The results are illustrated through numerical experiments.
}

\section{Introduction}

This paper studies a family of metrics between
real-variable functions, exploring their stability
to geometric deformations and additive noise.
The definitions of these metrics makes use of a function's ``scale-space''
representation, in which the function is
convolved with a sequence of Gaussian kernels of varying widths.
Each function is thereby replaced by a stack
of functions, of varying resolution.
The distance $\D_{\alpha,p,r}(f,g)$ between two functions $f$ and $g$
in $L^r \cap L^1$ is then defined as
a weighted sum of the distances across each stack;
the precise definition is given in Section \ref{sec:metrics_defined}.

Such metrics, and the ideas underlying their construction, have previously appeared in
numerous works:

\begin{itemize}

\item
In the harmonic analysis literature,
the norms that induce these metrics are
equivalent to the dual norms of Besov spaces
(certain classes of smooth functions)
\cite{triebel1983theory, sawano2018theory}.

\item
Classical imaging processing techniques make use
of an image's
scale-space representation via Gaussian smoothing
\cite{lindeberg1994scale}.
For instance, features extracted from the scale-space representation
of a function are used in
the well-known SIFT algorithm
for feature detection
\cite{lowe2004distinctive, lowe1999object}.

\item
The special case where $p=r=1$
and $0 < \alpha < 1$
is known to be metrically equivalent to
the Wasserstein distance
with ground distance
$|x-y|^{\alpha}$;
see \cite{leeb2016holder},
as well as the related works
\cite{jacobs2008approximate, craig2015wavelet},
which employ multiscale analysis via wavelets
rather than Gaussians.

\item
Analogous metrics over more abstract spaces
that exhibit a multiscale structure
(such as spaces endowed with hierarchical tree metrics)
have been studied in data science
applications; for instance, see
\cite{leeb2018mixed, leeb2016holder, leeb2018approximating, mishne2017datadriven,
mishne2016hierarchical, lin2026joint, charikar2002similarity}.

\end{itemize}

In the present work, we explore the stability of these
metrics $\D_{\alpha,p,r}$ to distortions of their inputs,
both under measure-preserving geometric deformations
and additive noise. We focus specifically on applications to image-processing
(i.e.\ to functions on  $\R^2$, and their discretizations),
though many of our results may be generalized to
other dimensions.
One of the scientific motivations for our work
is the robust comparison of 2D tomographic projections
of 3D functions, a setting which occurs in
applications like cryo-electron microscopy (cryo-EM)
\cite{bendory2020single, singer2020computational, van2000single, doerr2016single, vulovic2013image}.
Various metrics related to Wasserstein distances have been
proposed for image comparison in cryo-EM recently
\cite{shi2025fast, rao2020wasserstein},
not to mention scores of other applications
and theoretical results
\cite{rubner2000earth, rabin2011wasserstein, bonneel2015sliced,
jacobs2008approximate, levina2001earth, villani2003topics, villani2008optimal,
santambrogio2015optimal, peyre2019computational}.
Because the metrics $\D_{\alpha,p,r}$ generalize
a metric equivalent to Wasserstein, it
is natural to explore their properties as well.

Because the images in cryo-EM typically
undergo random in-plane rotations,
it is important to define a rotation-invariant version
of the metric $\D_{\alpha,p,r}$,
which minimizes the distance between the input
functions over all in-plane rotations,
as is done in \cite{shi2025fast} and  \cite{rao2020wasserstein}.
As will become clear later,
for this to be computationally feasible
we restrict $r=2$;
see Section \ref{sec:discretization_rotations} for details.

The main contributions of the paper are as follows:

\begin{itemize}

\item
We bound the distortion between $\D_{\alpha,p,r}(f,g)$
when $g$ is a perturbation of $f$, in terms
of the maximum displacement of the perturbation.
(In fact, the bound extends to the more general setting
where $f$ and $g$ are projections of functions
that are perturbations of each other.)

\item
If $f$ and $g$ are 2D projections of
a fixed 3D function,
we prove robustness of
the rotationally-invariant metric
$\D_{\alpha,p,r}(f ,g; \,\SO(2))$
to changes in the projection angle.

\item
When $r=2$, we define a fast algorithm for calculating
both $\D_{\alpha,p}(f,g)$ and the rotationally-invariant
metric $\D_{\alpha,p}(f,g;\,  \SO(2))$. For functions discretized
on an $n$-by-$n$ Cartesian grid, the cost of these algorithms
scale like $O(n^2 \log(n))$.

\item
We prove robustness of the discretization to additive sub-Gaussian noise.

\item
Finally, we report on numerical experiments illustrating the
aforementioned results.

\end{itemize}

\section{The scale-space metrics}
\label{sec:metrics_defined}

The metrics will be defined by several user-set parameters:
$\tau_0 > 0$, $0 < \omega < 1$, $r \ge 1$, $p \ge 1$, and $\alpha > 0$.
Typically, we will take $\omega = 1/2$;
and when rotational-invariance is needed, $r=2$.

For a value $\tau > 0$, denote by $G_\tau$ the Gaussian kernel at scale $\tau$,
defined by
\begin{align}
G_\tau(x) = \frac{1}{(\pi \tau)^{d/2}} e^{-|x|^2 / \tau}.
\end{align}
For fixed values $\tau_0 > 0$ and $0 < \omega < 1$, and for $k \ge 0$, define
\begin{align}
\tau_k = \tau_0 \omega^k.
\end{align}
Fix $\alpha > 0$,  $p \ge 1$, and $r \ge 1$. For a function $f$ in $L^r(\R^d)$,
we define the norm
\begin{align}
\|f\|_{\Gamma_{p,r}^\alpha}
= \left(\sum_{k \ge 0} \tau_k^{p \alpha / 2} \|G_{\tau_k} \ast f\|_{L^r}^p \right)^{1/p},
\end{align}
and denote the induced distance between functions $f$ and $g$ as follows:
\begin{align}
\D_{\alpha,p,r}(f,g) = \|f - g\|_{\Gamma_{p,r}^\alpha}.
\end{align}
When $r=2$, we will denote this more simply by $\D_{\alpha,p}(f,g)$;
and when $r=2$ and $p=1$, we will denote it by $\D_{\alpha}(f,g)$.

Given a subgroup $\calG \le O(d)$, we define the $\calG$-invariant
distance between $f$ and $g$ by
\begin{align}
\D_{\alpha,p,r}(f,g; \,\mathcal{G})
= \min_{U \in \mathcal{G}} \|f - g \circ U\|_{\Gamma_{p,r}^\alpha}.
\end{align}
$\D_{\alpha,p,r}(f,g; \,\mathcal{G})$ is really a distance between
the orbits of $f$ and $g$ under the action of $\calG$.
In this work, we will only consider the case where $\calG = \SO(2)$,
for comparing functions of two variables.

\subsection{Difference-of-Gaussians}

It is known that the norm $\|f\|_{\Gamma_{p,r}^\alpha}$
is in fact equivalent (in the sense that the ratios
are bounded above and below be constants independent of $f$)
to the norm
\begin{align}
\|f\|_{\Gamma_{p,r}^\alpha}^{(1)}
= \left( \tau_0^{p\alpha/2} \|G_{\tau_0} \ast f\|_{L^r}^p+ 
\sum_{k \ge 0} \tau_k^{p \alpha / 2} \|(G_{\tau_k}  - G_{\tau_{k+1}})\ast f\|_{L^r}^p \right)^{1/p}.
\end{align}
More precisely, we have the following result:

\begin{prop}
\label{prop:dog_equivalence}
There is a universal constant $C>0$
such that 
for all $f$ in $L^r$,
\begin{align}
\frac{1}{C} \cdot \|f\|_{\Gamma_{p,r}^\alpha}^{(1)} \le \|f\|_{\Gamma_{p,r}^\alpha}
\le \frac{C}{1 - \omega^{\alpha/2}} \cdot \|f\|_{\Gamma_{p,r}^{\alpha}}^{(1)} .
\end{align}
\end{prop}

\begin{rmk}
Results of this form are well known in classical Besov space theory;
indeed, there are many equivalent formulas for Besov norms
and their duals \cite{triebel1983theory, sawano2018theory}.
Though there is nothing essentially novel in the result, for the reader's convenience,
we give a self-contained, elementary proof.
\end{rmk}

\begin{rmk}
The norm $\|f\|_{\Gamma_{p,r}^\alpha}^{(1)}$ is of interest in that it convolves
with the difference-of-Gaussian filters $G_{\tau_k}  - G_{\tau_{k+1}}$;
that is, it applies a sequence of band-pass filters
to $f$, rather than low-pass filters.
The difference-of-Gaussians filters are another classical
tool of image processing, useful for edge detection
\cite{marr1980theory, wilson1977threshold}.
Related ideas underly the Laplacian pyramids
algorithm \cite{burt1983laplacian}, among others.
\end{rmk}

\begin{rmk}
Note that upper bound on
$\|f\|_{\Gamma_{p,r}^{\alpha}} / \|f\|_{\Gamma_{p,r}^{\alpha}}^{(1)}$
grows to infinity as $\alpha \to 0^+$,
whereas the lower bound is a universal constant.
This behavior is to be expected,
and understanding why may give some intuition
as to the way the norms are defined.
Reasoning informally,
each term $G_{\tau_k} \ast f$ appearing
in $\|f\|_{\Gamma_{p,r}^{\alpha}}$
serves as a low-pass filter,
which essentially preserves the frequencies $\xi$
in the range $0 \le |\xi| \le A / \sqrt{\tau_k}$
(for some $A>0$,
with leakage, that we will ignore, due to the non-compact support of $\what{G}_{\tau_k}$).
On the other hand, each term $(G_{\tau_k} - G_{\tau_{k+1}})\ast f$
appearing in $\|f\|_{\Gamma_{p,r}^{\alpha}}^{(1)}$
serves as a band-pass filter,
preserving the frequencies $\xi$ of $f$
in the range $A/\sqrt{\tau_k} < |\xi| < A/\sqrt{\tau_{k+1}}$
(again, ignoring leakage outside this range).
In other words, each summand of $\|f\|_{\Gamma_{p,r}^{\alpha}}^{(1)}$
effectively
measures $f$'s frequency content in the \emph{disjoint}
regions $A/\sqrt{\tau_k} < |\xi| < A/\sqrt{\tau_{k+1}}$,
whereas each summand of $\|f\|_{\Gamma_{p,r}^{\alpha}}$
measures the frequency content in the \emph{overlapping}
regions $0 \le |\xi| < A/ \sqrt{\tau_k}$;
the inequality $\|f\|_{\Gamma_{p,r}^{\alpha}}^{(1)} \le C \cdot \|f\|_{\Gamma_{p,r}^{\alpha}}$
is therefore not surprising, since the series
defining $\|f\|_{\Gamma_{p,r}^{\alpha}}$
adds up many redundant contributions from overlapping
regions. On the other hand,
for each $k$,
though the terms $\|G_{\tau_\ell} \ast f\|_{L^r}^p$,
$\ell \ge k$,
all capture frequencies from the region $A/\sqrt{\tau_k} < |\xi| < A/\sqrt{\tau_{k+1}}$,
the geometrically decreasing weights $\tau_\ell^{p \alpha /2}$
ensure that the weighted sum of
all these contributions is still proportional
to the first term $\|G_{\tau_k} \ast f\|_{L^r}^p$,
with a constant the grows as the weights $\tau_\ell^{p \alpha /2}$
decay slower, i.e. as $\alpha \to 0^+$.
\end{rmk}

It will first be convenient to state an elementary lemma,
which we prove for the reader's convenience.

\begin{lem}
\label{lem:dumb_bound-2}
Let $0 < \delta < 1$.
Then
\begin{align}
\left( \frac{1}{1 - \delta^{p}} \right)^{1/p} \le \frac{1}{1-\delta}
\end{align}
for all $p \ge 1$.

\end{lem}

\begin{proof}[Proof of Lemma \ref{lem:dumb_bound-2}]

The inequality is equivalent to
\begin{align}
1 - \delta^p - (1-\delta)^p \ge 0.
\end{align}
Equality holds when $p=1$, and since $0 < \delta < 1$, the right side is increasing
as a function of $p$ (since both $\delta^p$ and $(1-\delta)^p$
decrease as $p$ grows), proving the result.

\end{proof}

\begin{proof}[Proof of Proposition \ref{prop:dog_equivalence}]

First,
using the inequality
\begin{align}
\|(G_{\tau_{k+1}}  - G_{\tau_k})\ast f\|_{L^r}^{p}
\le 2^p \left(\|G_{\tau_{k+1}} \ast f \|_{L^r}^{p} + \| G_{\tau_k}\ast f\|_{L^r}^{p} \right)
\end{align}
it immediately follows that
\begin{align}
\| f \|_{\Gamma_{p,r}^\alpha}^{(1)} \le C \cdot \| f \|_{\Gamma_{p,r}^{\alpha}},
\end{align}
for a universal constant $C>0$.

For the reverse direction, from the telescoping sum
\begin{align}
G_{\tau_k} \ast f
= G_{\tau_0} \ast f + \sum_{\ell=1}^{k} (G_{\tau_\ell} - G_{\tau_{\ell-1}}) \ast f
\end{align}
it follows that
\begin{align}
\|G_{\tau_k} \ast f\|_{L^r}
\le \|G_{\tau_0} \ast f\|_{L^r}
    + \sum_{\ell=1}^{k} \|(G_{\tau_\ell} - G_{\tau_{\ell-1}}) \ast f\|_{L^r}.
\end{align}
Consequently,
\begin{align}
\label{eq:604002}
\tau_k^{\alpha/2} \|G_{\tau_k} \ast f\|_{L^r}
&\le \tau_k^{\alpha/2} \|G_{\tau_0} \ast f \|_{L^r}
    + \sum_{\ell=1}^{k} \tau_k^{\alpha/2} \tau_{\ell}^{-\alpha/2} \tau_{\ell}^{\alpha/2}
        \|(G_{\tau_\ell} - G_{\tau_{\ell-1}}) \ast f\|_{L^r}
\nonumber \\
&= \tau_k^{\alpha/2} \|G_{\tau_0} \ast f \|_{L^r}
    + \sum_{\ell=1}^{k} \omega^{(k-\ell)\alpha/2} \cdot \tau_{\ell}^{\alpha/2}
        \|(G_{\tau_\ell} - G_{\tau_{\ell-1}}) \ast f\|_{L^r}.
\end{align}
Define $A[m] = \omega^{m \alpha/2}$ for $m \ge 0$, and $A[m]=0$ for $m < 0$;
and define $B[\ell] = \tau_\ell^{\alpha/2} \|(G_{\tau_\ell} - G_{\tau_{\ell-1}}) \ast f\|_{L^r}$
for $\ell>0$, and $B[\ell]=0$ for $\ell \le 0$.
Then the series on the right side of \eqref{eq:604002}
may be written as a convolution between $A$ and $B$ over $\ZZ$:
\begin{align}
\sum_{\ell=1}^{k} \omega^{(k-\ell)\alpha/2} \cdot \tau_{\ell}^{\alpha/2}
    \|(G_{\tau_\ell} - G_{\tau_{\ell-1}}) \ast f\|_{L^r}
= \sum_{\ell=-\infty}^{\infty} A[k-\ell] B[\ell]
= (A \ast B)[k].
\end{align}
Note that $(A \ast B)[k] = 0$ if $k \le 0$,
since in that case all the terms $A[k-\ell] = 0$ for $\ell > 0$,
and $B[\ell] = 0$ for $\ell \le 0$.
From
Young's inequality (see, for instance, Theorem 1.2.12
in \cite{grakafos2010classical}),
summing over $k$ in \eqref{eq:604002} gives the bound
\begin{align}
\left(\sum_{k=0}^\infty
        \left(\sum_{\ell=1}^{k} \tau_k^{\alpha/2} \tau_{\ell}^{-\alpha/2} \tau_{\ell}^{\alpha/2}
        \|(G_{\tau_\ell} - G_{\tau_{\ell-1}}) \ast f\|_{L^r} \right)^p \right)^{1/p}
&\le \left(\sum_{k=0}^\infty |(A \ast B) [k]|^p \right)^{1/p}
\nonumber \\
&\le \left(\sum_{k=-\infty}^{\infty} A[k] \right)
    \cdot \left( \sum_{\ell=-\infty}^{\infty} B[\ell]^{p} \right)^{1/p}
\nonumber \\
&= \left(\sum_{k=0}^{\infty} \omega^{k\alpha/2} \right)
    \cdot \left( \sum_{\ell=1}^{\infty} \tau_\ell^{p\alpha/2}
        \|(G_{\tau_\ell} - G_{\tau_{\ell-1}}) \ast f\|_{L^r}^p \right)^{1/p}
\nonumber \\
&= \left(\frac{1}{1 - \omega^{\alpha/2}} \right)
    \cdot \left( \sum_{\ell=1}^{\infty} \tau_\ell^{p\alpha/2}
        \|(G_{\tau_\ell} - G_{\tau_{\ell-1}}) \ast f\|_{L^r}^p \right)^{1/p}
\end{align}

Consequently, since
\begin{align}
\|G_{\tau_k} \ast f\|_{L^r}
\le \|G_{\tau_0} \ast f\|_{L^r}
    + \sum_{\ell=1}^{k} \|(G_{\tau_\ell} - G_{\tau_{\ell-1}}) \ast f\|_{L^r},
\end{align}
and hence
\begin{align}
\|G_{\tau_k} \ast f\|_{L^r}^p
\le 2^p \|G_{\tau_0} \ast f\|_{L^r}^p
    + 2^p \left(\sum_{\ell=1}^{k} \|(G_{\tau_\ell} - G_{\tau_{\ell-1}}) \ast f\|_{L^r}\right)^p,
\end{align}
it follows that
\begin{align}
\sum_{k=0}^{\infty} \tau_k^{p \alpha / 2} \|G_{\tau_k} \ast f\|_{L^r}^p
&\le 2^p \sum_{k=0}^{\infty} \tau_k^{p\alpha/2}\|G_{\tau_0} \ast f\|_{L^r}^p
    + 2^p\sum_{k=0}^\infty
        \left(\sum_{\ell=1}^{k} \tau_k^{\alpha/2} \tau_{\ell}^{-\alpha/2} \tau_{\ell}^{\alpha/2}
        \|(G_{\tau_\ell} - G_{\tau_{\ell-1}}) \ast f\|_{L^r} \right)^p
\nonumber \\
&\le 2^p \frac{\tau_0^{p\alpha/2}}{1 - \omega^{p\alpha/2}}\|G_{\tau_0} \ast f\|_{L^r}^p
    + 2^p \left(\frac{1}{1 - \omega^{\alpha/2}} \right)^p
    \cdot \left( \sum_{\ell=1}^{\infty} \tau_\ell^{p\alpha/2}
        \|(G_{\tau_\ell} - G_{\tau_{\ell-1}}) \ast f\|_{L^r}^p \right),
\end{align}
and so
\begin{align}
\|f\|_{\Gamma_{p,r}^\alpha}
&= \left( \sum_{k=0}^{\infty} \tau_k^{p \alpha / 2} \|G_{\tau_k} \ast f\|_{L^r}^p \right)^{1/p}
\nonumber \\
&\le  2\left(\frac{1}{1 - \omega^{p\alpha/2}} \right)^{1/p}
        \tau_0^{\alpha/2} \|G_{\tau_0} \ast f\|_{L^r}
    + \left(\frac{2}{1 - \omega^{\alpha/2}} \right)
    \cdot \left( \sum_{\ell=1}^{\infty} \tau_\ell^{p\alpha/2}
        \|(G_{\tau_\ell} - G_{\tau_{\ell-1}}) \ast f\|_{L^r}^p \right)^{1/p}
\nonumber \\
&\le  \left(\frac{2}{1 - \omega^{\alpha/2}} \right)
        \tau_0^{\alpha/2} \|G_{\tau_0} \ast f\|_{L^r}
    + \left(\frac{2}{1 - \omega^{\alpha/2}}  \right)
    \cdot \left( \sum_{\ell=1}^{\infty} \tau_\ell^{p\alpha/2}
        \|(G_{\tau_\ell} - G_{\tau_{\ell-1}}) \ast f\|_{L^r}^p \right)^{1/p}
\nonumber \\
&\le \frac{C}{1 - \omega^{\alpha/2}} \cdot \|f\|_{\Gamma_{p,r}^{\alpha}}^{(1)},
\end{align}
as claimed.

\end{proof}

\subsection{Choice of $\omega$}

The norms $\|f\|_{\Gamma_{p,r}^{\alpha}}$,
and corresponding distances $\D_{\alpha,p,r}(f,g)$,
depend on a choice of rate parameter $0 < \omega < 1$
(which we suppress from the notation).
It is therefore natural to ask about the relationships
between these norms/distances for different choices
of $\omega$.
It is easy to see
that all choices of $0 < \omega < 1$ yield equivalent norms,
in the sense that their ratios are bounded above and below
by values independent of $f$;
such a proof can be extracted from any
standard reference on Besov spaces \cite{sawano2018theory, triebel1983theory}.
Indeed, the argument follows easily by
showing that for any choice of $\omega$, the norm $\|f\|_{\Gamma_{p,r}^\alpha}$
is equivalent to the continuous-scale norm
\begin{align}
\left(\int_{0}^{\tau_0} t^{p\alpha/2- 1} \|G_t \ast f\|_{L^r}^p \, dt \right)^{1/p}.
\end{align}
This may be shown by breaking this integral
into integrals over the subintervals $[\tau_{k+1},\tau_{k}]$
and obtaining a suitable upper or lower bound
on each subinterval
using monotonicity of the functions $t^{p \alpha/2-1}$
and $\|G_t \ast f\|_{L^r}^p$
(viewed as functions of $t$).
We omit the straightforward details.

\section{Key definitions}

\subsection{Deformations and maximum displacement}

Suppose $A$ and $B$ are open subsets of $\R^d$,
and $\varphi : B \to A$ is a $C^1$ bijection.
If $f : A \to \R$ is in $L^1$, then $\varphi$ induces
a new function $f_\varphi : B \to \R$ which we define as
\begin{align}
f_\varphi(x) = f(\varphi(x)) |\nabla \varphi (x)|,
\end{align}
where $\nabla \varphi(x)$ is the Jacobian matrix of $\varphi$,
and $|\nabla \varphi(x)|$ denotes the absolute value of its determinant.
We will call $f_\varphi$  a deformation of $f$
(and sometimes by abuse of terminology we may also
refer to $\varphi$ itself as a deformation).
Note that $f$ and $f_\varphi$ have the same integrals
and $L^1$ norms. Note too that if $\psi = \varphi^{-1}$
is the functional inverse of $\phi$, then $f = (f_\varphi)_{\psi}$.

We define the maximum displacement of $\varphi$ by
\begin{align}
\varepsilon(\varphi) = \max_{x \in B} |x - \varphi(x)|.
\end{align}
Note too that $\varepsilon(\varphi) = \varepsilon(\varphi^{-1})$.
This definition of maximum displacement is used in,
for instance, the paper \cite{leeb2025sliced},
and will be how we measure the ``size'' of the deformation $\varphi$.

\begin{rmk}
The deformation $f_\varphi$ preserves the $L^1$
norm and integral of $f$. This is natural in our setting,
due to the relationship between the metrics $\D_{\alpha,p,r}$
and Wasserstein distances (which are defined between probability
measures, which obviously have equal measure).
We do note, however, that other works have considered
deformations of the form $f(\varphi(x))$, which instead preserve the
$L^\infty$ norm of $f$; see \cite{mallat2012group, anden2014deep}.

\end{rmk}

\subsection{Mixed norms and projections}

For a subspace $\calU \le \R^D$ of dimension $d$,
we define the projection $\calP_\calU$ onto $\calU$ by
\begin{align}
(\calP_\calU f)(x) = \int_{\calU^{\perp}} f(x,y) \,dy.
\end{align}
Projections of this type are fundamental in tomography
\cite{herman2009fundamentals, bendory2020single, singer2020computational, van2000single, doerr2016single, vulovic2013image}.

For a subspace $\calU$ of $\RR^D$
of dimension $d < D$,
we define the mixed norm
\begin{align}
\|f\|_{M^r(\calU)}
= \left(\int_{\calU} \left(\int_{\calU^\perp} |f(x,y)| dy \right)^r dx \right)^{1/r}.
\end{align}
Note that
\begin{align}
\|f\|_{M^r(\calU)} = \|\calP_\calU(|f|)\|_{L^r(\calU)}.
\end{align}

We also define the max-mixed norm
\begin{align}
\|f\|_{M_d^r}
= \sup_{\dim(\calU) = d} \|f\|_{M^r(\calU)}.
\end{align}
When $d$ is clear from the context, we will just denote this
by $\|f\|_{M^r}$.

\begin{rmk}
Note that $\|f\|_{M^{r}}$ is rotationally-invariant,
and that if $f$ is compactly-supported, $\|f\|_{M^{r}}$ is controlled
by the $L^r$ norm of $f$. Indeed, if $f$ is continuous and supported on, say,
the ball $B_D(0,R)$ of radius $R$ in $\R^D$, then
each $f(x,\cdot)$ is supported on the ball $B_{D-d}(0,R)$ of radius $R$
in $\calU^{\perp}$ (identified with $\R^{D-d}$),
and denoting its volume by $V_{D-d}(R)$,
\begin{align}
\calP_\calU(|f|) (x) = \int_{\calU^{\perp}} |f(x,y)| \, dy
\le V_{D-d}(R)^{1 - 1/r} \left(\int_{\calU^{\perp}} |f(x,y)|^r \, dy\right)^{1/r}.
\end{align}
It then follows that
\begin{align}
\|f\|_{M^r(\calU)}
= \|\calP_\calU (|f|)\|_{L^r(\calU)}
\le V_{D-d}(R)^{1 - 1/r} \left(\int_{\calU} \int_{\calU^{\perp}} |f(x,y)|^r \, dy \, dx\right)^{1/r}
= V_{D-d}(R)^{1 - 1/r} \|f\|_{L^r},
\end{align}
and taking the supremum over $\calU$ shows $\|f\|_{M^r} \le V_{D-d}(R)^{1 - 1/r} \|f\|_{L^r}$
as well.

\end{rmk}

\section{Properties of $\D_{\alpha,p,r}$}

\subsection{Stability under deformations}

In this section, we establish bounds on the distances
$\D_{\alpha,p,r}(f,f_\varphi)$
between a function $f$ in $L^r$ and its deformation
$f_\varphi(x) = f(\varphi) |\nabla \varphi(x)|$.

The simplest bound is the following:
\begin{prop}
\label{prop:trivial}
For any parameters $\alpha,p,r,\omega,\tau_0$,
and any functions $f$ and $g$ in $L^r$,
\begin{align}
\D_{\alpha,p,r}(f,g)
\le \frac{\tau_0^{\alpha / 2}}{1 - \omega^{\alpha/2}} \cdot \|f-g\|_{L^r}.
\end{align}

\end{prop}

\begin{proof}
Using Young's inequality for convolutions (e.g.\ Theorem 8.7
in \cite{folland1999real})
and $\|G_{\tau}\|_{L^1} = 1$,
\begin{align}
\D_{\alpha,p,r}(f,g)^p
= \sum_{k \ge 0} \tau_k^{p \alpha / 2} \|G_{\tau_k} \ast (f-g)\|_{L^r}^p
\le \|f-g\|_{L^r}^p \sum_{k \ge 0} \tau_k^{p \alpha / 2}
= \frac{\tau_0^{p \alpha / 2}}{1 - \omega^{p\alpha/2}} \|f-g\|_{L^r}^p,
\end{align}
and the result follows by taking $p$-th roots and using Lemma \ref{lem:dumb_bound-2}.

\end{proof}

Consequently, the metric $\D_{\alpha,p,r}(f,g)$ is no less stable than $\|f-g\|_{L^r}$.
However, we now state and prove a bound
that controls the distance $\D_{\alpha,p,r}(f,f_\varphi)$
by the deformation size $\varepsilon(\varphi)$.
Similar bounds are known (and are essentially tautological)
for Wasserstein distances, and are also known for other distances
such as the Cram\'er and sliced Cram\'er metrics \cite{leeb2025sliced}.

For $\alpha > 0$ and $p \ge 1$,
define the function $\Delta_{\alpha,p}(\epsilon)$, $0 < \epsilon < 1$, as follows:
\begin{align}
\Delta_{\alpha,p}(\epsilon)=
\begin{cases}
\begin{aligned}
\frac{\epsilon^\alpha}{(\alpha - \alpha^2)^{1/p}},
\end{aligned}
    & \text{ if } 0 < \alpha < 1, \\ \\
\begin{aligned}
\epsilon \cdot \left(\log\left(1/\epsilon\right) + 1 \right)^{1/p},
\end{aligned}
    & \text{ if } \alpha = 1, \\ \\
\begin{aligned}
\frac{\epsilon}{\min\{(\alpha - 1)^{1/p} , 1\}},
\end{aligned}
    & \text{ if } \alpha > 1.
\end{cases}
\end{align}

\begin{thm}
\label{thm:main_projections}
Let $A, B \subset \R^D$ be open sets,
and let $\varphi: B \to A$ be a $C^1$ bijection.
Let $d \le D$ and $\calU$ be a $d$-dimensional subspace of $\RR^D$.
Suppose $f$ is in $L^1 \cap L^r$ and supported on $A$, with
$\|f_\varphi\|_{M^r(\calU)} \le \|f\|_{M^r(\calU)} < \infty$,
where $r \ge 1$.
Let $\varepsilon = \varepsilon(\varphi)= \max_{x \in B}|x - \varphi(x)|$,
and suppose $\varepsilon < \sqrt{\tau_0}$.
Then for all $\alpha > 0$ and $p \ge 1$,
\begin{align}
\D_{\alpha,p,r}(\calP_\calU f,\calP_\calU f_\varphi)
\le C \cdot \|f\|_{M^r(\calU)} \cdot \frac{\tau_0^{\alpha/2}}{1 - \omega^{1/2}}
    \cdot \Delta_{\alpha,p}(\varepsilon / \sqrt{\tau_0})
\end{align}
where $C>0$ depends only on the dimension $d$.

\end{thm}

As an immediate corollary without projections, we have:
\begin{cor}
Let $A, B \subset \R^d$ be open sets,
and let $\varphi: B \to A$ be a $C^1$ bijection.
Suppose $f$ is in $L^r$ and supported on $A$, and satisfies
$\|f_\varphi\|_{L^r} \le \|f\|_{L^r} < \infty$,
where $r \ge 1$.
Let $\varepsilon = \varepsilon(\varphi)= \max_{x \in B}|x - \varphi(x)|$,
and suppose $\varepsilon < \sqrt{\tau_0}$.
Then for all $\alpha > 0$ and $p \ge 1$,
\begin{align}
\D_{\alpha,p,r}(f, f_\varphi)
\le C \cdot \|f\|_{L^r} \cdot \frac{\tau_0^{\alpha/2}}{1 - \omega^{1/2}}
    \cdot \Delta_{\alpha,p}(\varepsilon / \sqrt{\tau_0})
\end{align}
where $C>0$ depends only on the dimension $d$.

\end{cor}

The proof of Theorem \ref{thm:main_projections}
follows from the following bounds on
the distances at each sufficiently coarse level:

\begin{prop}
\label{prop:bound_layer}
Adopt the notation of Theorem \ref{thm:main_projections},
let $\tau > 0$, and suppose $\varepsilon = \varepsilon(\varphi) \le \sqrt{\tau}$.
Then
\begin{align}
\|G_\tau \ast (\calP_\calU f - \calP_\calU f_\varphi)\|_{L^r}
\le C \cdot \|f\|_{M^r(\calU)} \cdot \frac{\varepsilon}{\sqrt{\tau}},
\end{align}
where $C>0$ depends on the dimension $d$.
\end{prop}

We will make use of the following lemma:

\begin{lem}
\label{lem:bound_xi}
Suppose $0 < \delta \le 1$, and $|\xi(y)| \le \delta$
for all $y$ in $\R^d$. Let $G = G_1$. Then
\begin{align}
\int_{\R^d} |G(y) - G(y + \xi(y))| dy
\le 2\frac{\Gamma((d+1)/2)}{\Gamma(d/2)} \delta + O(\delta^2).
\end{align}
\end{lem}

\begin{proof}[Proof of Lemma \ref{lem:bound_xi}]

We have the factorization
\begin{align}
G(y + \xi(y))
&= \frac{1}{\pi^{d/2}} e^{-|y + \xi(y)|^2}
\nonumber \\
&= \frac{1}{\pi^{d/2}} e^{-|y|^2} e^{-(2\la y,\xi(y)\ra +  |\xi(y)|^2)}
\nonumber \\
&= G(y) e^{-|\xi(y)|^2} e^{-2\la y,\xi(y)\ra},
\end{align}
and so
\begin{align}
\int |G(y) - G(y + \xi(y))| dy
&= \int \left| G(y) e^{-|\xi(y)|^2} e^{|\xi(y)|^2}
    - G(y) e^{-|\xi(y)|^2} e^{-2\la y,\xi(y)\ra} \right| dy
\nonumber \\
&= \int G(y) e^{-|\xi(y)|^2}  \left| e^{|\xi(y)|^2} - e^{-2\la y,\xi(y)\ra} \right| dy
\nonumber \\
&\le \int G(y) \left| e^{|\xi(y)|^2} - e^{-2\la y,\xi(y)\ra} \right| dy
\nonumber \\
&= \int G(y) \left| e^{|\xi(y)|^2} - 1 + 1 - e^{-2\la y,\xi(y)\ra} \right| dy
\nonumber \\
&\le \int G(y) \left| e^{|\xi(y)|^2} - 1 \right| dy
    + \int G(y) \left| e^{-2\la y,\xi(y)\ra} - 1\right| dy.
\end{align}

For the first integral, since $|\xi(y)| \le \delta \le 1$,
there is a universal constant $C > 0$ such that for all $y$,
\begin{align}
\label{eq:bound4030201}
e^{|\xi(y)|^2} - 1 = C |\xi(y)|^2
\le C \delta^2,
\end{align}
and so, since $\int G(y) dy = 1$,
\begin{align}
\int G(y) \left| e^{|\xi(y)|^2} - 1 \right| dy
\le C  \delta^2.
\end{align}

For the second integral: we have the Taylor expansion
\begin{align}
e^{-2\la y,\xi(y)\ra} - 1
= \sum_{k=1}^{\infty} \frac{(-2\la y,\xi(y)\ra)^k}{k!},
\end{align}
and so, since $|\la y,\xi(y)\ra| \le |y| |\xi(y)| \le \delta |y|$,
\begin{align}
\left| e^{-2\la y,\xi(y)\ra} - 1 \right|
\le \sum_{k=1}^{\infty} \frac{(2 \delta)^k }{k!} |y|^k.
\end{align}

Now, for each $k \ge 1$, changing into spherical coordinates
and using that the area of the sphere of radius $r$ in $\R^d$
is
\begin{align}
A_{d-1}(r) = r^{d-1} \frac{2 \pi^{d/2}}{\Gamma(d/2)},
\end{align}
we get
\begin{align}
\int G(y) |y|^k dy
&= \int \frac{e^{-|y|^2}}{\pi^{d/2}} |y|^k dy
\nonumber \\
&= \frac{2 \pi^{d/2}}{\Gamma(d/2)}
    \int_{0}^{\infty} r^{d-1} r^k \frac{e^{-r^2}}{\pi^{d/2}} dr
\nonumber \\
&= \frac{2}{\Gamma(d/2)}
    \int_{0}^{\infty} r^{d+k-1} e^{-r^2}dr,
\nonumber \\
&= \frac{2}{\Gamma(d/2)}
    \int_{0}^{\infty} r^{d+k-2} e^{-r^2} rdr,
\end{align}
and using the change-of-variables $s = r^2$, $ds = 2rdr$,
\begin{align}
\int G(y) |y|^k dy
&= \frac{2}{\Gamma(d/2)}
    \int_{0}^{\infty} r^{d+k-2} e^{-r^2} rdr
\nonumber \\
&= \frac{1}{\Gamma(d/2)}
    \int_{0}^{\infty} s^{(d+k)/2 - 1} e^{-s} ds
\nonumber \\
&= \frac{\Gamma((d+k)/2)}{\Gamma(d/2)}.
\end{align}

Consequently, we have the Taylor series
\begin{align}
\label{eq:bound4030201-1}
\int G(y) \left| e^{-2\la y,\xi(y)\ra} - 1\right| dy
&\le \sum_{k=1}^{\infty} \frac{(2 \delta)^k }{k!} \int G(y) |y|^k dy
\nonumber \\
&= \sum_{k=1}^{\infty} \frac{\Gamma((d+k)/2)}{\Gamma(d/2) k!} (2\delta)^k
\nonumber \\
&= 2\frac{\Gamma((d+1)/2)}{\Gamma(d/2)} \delta + O(\delta^2).
\end{align}

Combining \eqref{eq:bound4030201-1} with \eqref{eq:bound4030201}
gives the desired bound
\begin{align}
\int_{\R^d} |G(y) - G(y + \xi(y))| dy
\le 2\frac{\Gamma((d+1)/2)}{\Gamma(d/2)} \delta + O(\delta^2).
\end{align}

\end{proof}

\begin{proof}[Proof of Proposition \ref{prop:bound_layer}]

We recall that we wish to show the bound
\begin{align}
\|G_\tau \ast (\calP_\calU f - \calP_\calU f_\varphi)\|_{L^r}
\le C \cdot \|f\|_{M^r(\calU)} \cdot \frac{\varepsilon}{\sqrt{\tau}},
\end{align}
for $\varepsilon \le \sqrt{\tau}$.
Without loss of generality, we identify the $d$-dimensional subspace $\calU$
with $\R^d$; and for brevity, we let $\calP = \calP_\calU$.

First, letting $\psi = \varphi^{-1}$, and $\psi(y) = (\psi_1, \psi_2) \in \R^d \times \R^{D-d}$,
\begin{align}
(G_\tau \ast \calP f_\varphi)(x)
&= \int_{\R^d} G_\tau(x-w) (\calP f_\varphi)(w) \, dw
\nonumber \\
&= \int_{\R^d} G_\tau(x-w) \int_{\R^{D-d}} f_\varphi(w,y) \,dy \, dw
\nonumber \\
&= \int_{\R^d} \int_{\R^{D-d}} G_\tau(x-w) f(\varphi(w,y)) |\nabla \varphi(w,y)| \,dy \, dw
\nonumber \\
&= \int_{\RR^d} \int_{\RR^{D-d}} G_\tau(x-\psi_1(u,v)) f(u,v) \, dv \,du
\end{align}
and similarly
\begin{align}
(G_\tau \ast \calP f)(x)
= \int_{\RR^d} \int_{\RR^{D-d}} G_\tau(x-u) f(u,v)  \, dv \,du .
\end{align}

Therefore,
\begin{align}
(G_{\tau} \ast [\calP f - \calP f_\varphi])(x)
&= \int_{\RR^d} \int_{\RR^{D-d}}
    \left[ G_\tau(x-u) - G_\tau(x-\psi_1(u,v)) \right] f(u,v) \, dv\, du,
\end{align}
and so
\begin{align}
\left| (G_{\tau} \ast [\calP f - \calP f_\varphi])(x) \right|
&\le \int_{\RR^d} \int_{\RR^{D-d}}
    \left| G_\tau(x-u) - G_\tau(x-\psi_1(u,v)) \right| \cdot |f(u,v)| \, dv\, du,
\nonumber \\
&\le \int_{\RR^d} \left(\int_{\RR^{D-d}} |f(u,v)| \, dv \right)
    \left(\sup_{v \in \R^{D-d}}\left| G_\tau(x-u) - G_\tau(x-\psi_1(u,v)) \right| \right) \, du
\nonumber \\
&= \int_{\RR^d} (\calP |f|)(u)
    \left(\sup_{v \in \R^{D-d}}\left| G_\tau(x-u) - G_\tau(x-\psi_1(u,v)) \right| \right) \, du.
\end{align}

Since $|\psi_1(u,v) - u| \le \varepsilon$,
\begin{align}
\sup_{v \in \RR^{D-d}}\left| G_\tau(x-u) - G_\tau(x-\psi_1(u,v)) \right|
\le \max_{w \in B(u,\varepsilon)} \left| G_\tau(x-u) - G_\tau(x-w) \right|,
\end{align}
noting that the maximum is achieved since the function being maximized
is continuous and the closed ball $B(u,\varepsilon)$ is compact;
we therefore define
\begin{align}
w^*(x,u)
= \argmax_{w \in B(u,\varepsilon)}\left| G_\tau(x-u) - G_\tau(x-w) \right|.
\end{align}

For brevity, let
\begin{align}
K_\tau(x,u)
= \left| G_\tau(x-u) - G_\tau(x-w^*(x,u)) \right|.
\end{align}
We then wish to bound the $L^r$ norm of the function
\begin{align}
x \mapsto \int_{\R^d} K_\tau(x,u) h(u) du,
\end{align}
where $h(u) = \calP |f|$.
We will follow a standard argument, by first bounding
the integrals $\int K_\tau(x,u) du$ and $\int K_\tau(x,u) dx$,
making use of Lemma \ref{lem:bound_xi}.

After a change of variables $y = (x-u)/\sqrt{\tau}$, i.e.\ $x = \sqrt{\tau} y+u$,
\begin{align}
\int K_\tau(x,u) dx = \int |G_\tau(x-u) - G_\tau(x - w^*(x,u))| dx
&= \int \left|G(y)
    - G\left(\frac{\sqrt{\tau}y + u - w^*(\sqrt{\tau}y+u,u)}{\sqrt{\tau}}\right)\right| dy
\nonumber \\
&= \int \left|G(y)
    - G\left(y+\frac{u - w^*(\sqrt{\tau}y+u,u)}{\sqrt{\tau}}\right)\right| dy
\nonumber \\
&= \int |G(y) - G(y + \xi(y))| dy,
\end{align}
where
\begin{align}
\xi(y) = \frac{u - w^*(\sqrt{\tau}y+u,u)}{\sqrt{\tau}}.
\end{align}
Note that, since $w^*(x,u) \in B(u,\varepsilon)$,
\begin{align}
|\xi(y)| \le \frac{\varepsilon}{\sqrt{\tau}}.
\end{align}
From Lemma \ref{lem:bound_xi}, therefore,
\begin{align}
\int |G_\tau(x-u) - G_\tau(x - w^*(x,u))| dx
\le 2\frac{\Gamma((d+1)/2)}{\Gamma(d/2)} \frac{\varepsilon}{\sqrt{\tau}}
    + O\left(\left( \frac{\varepsilon}{\sqrt{\tau}}\right)^2\right).
\end{align}

Next, we bound the other integral,
using a nearly identical argument.
Indeed, after a change of variables $y = (x-u)/\sqrt{\tau}$, i.e.\ $u = x - \sqrt{\tau} y$,
\begin{align}
\int K_\tau(x,u) du = \int |G_\tau(x-u) - G_\tau(x - w^*(x,u))| du
&= \int \left|G(y)
    - G\left(\frac{x - w^*(x,x-\sqrt{\tau}y)}{\sqrt{\tau}}\right)\right| dy
\nonumber \\
&= \int \left|G(y)
    - G\left(y+\frac{x-\sqrt{\tau}y - w^*(x,x-\sqrt{\tau}y)}{\sqrt{\tau}}\right)\right| dy
\nonumber \\
&= \int |G(y) - G(y + \xi(y))| dy,
\end{align}
where
\begin{align}
\xi(y) = \frac{x-\sqrt{\tau}y - w^*(x,x-\sqrt{\tau}y)}{\sqrt{\tau}}.
\end{align}
Note that, since $w^*(x,x-\sqrt{\tau}y) \in B(x-\sqrt{\tau}y,\varepsilon)$,
\begin{align}
|\xi(y)| \le \frac{\varepsilon}{\sqrt{\tau}}.
\end{align}
From Lemma \ref{lem:bound_xi}, therefore,
\begin{align}
\int |G_\tau(x-u) - G_\tau(x - w^*(x,u))| du
\le 2\frac{\Gamma((d+1)/2)}{\Gamma(d/2)} \frac{\varepsilon}{\sqrt{\tau}}
    + O\left(\left( \frac{\varepsilon}{\sqrt{\tau}}\right)^2\right).
\end{align}

It then follows (see, e.g.\ Theorem 6.18 in \cite{folland1999real})
that the $L^r$ operator norm of the kernel $K_\tau$
is bounded above
\begin{align}
\|K_\tau\|_{L^r \to L^r}
    \le 2\frac{\Gamma((d+1)/2)}{\Gamma(d/2)} \frac{\varepsilon}{\sqrt{\tau}}
    + O\left(\left( \frac{\varepsilon}{\sqrt{\tau}}\right)^2\right),
\end{align}
and as a consequence,
\begin{align}
\| G_\tau \ast [\calP f - \calP_\varphi]\|_{L^r}
\le \left( \left| \int (\calP |f|)(u) K_\tau(x,u) du \right|^r dx \right)^{1/r}
\le C \|\calP (|f|) \|_{L^r} \frac{\varepsilon}{\sqrt{\tau}}
= C \cdot \|f \|_{M^r(\calU)} \cdot \frac{\varepsilon}{\sqrt{\tau}},
\end{align}
as desired.

\end{proof}

We now employ Proposition \ref{prop:bound_layer}
to prove Theorem \ref{thm:main_projections}.
We first state an elementary lemma,
which we prove for the reader's convenience.

\begin{lem}
\label{lem:dumb_bound}
For all $0 < \eta < 1$ and
for all $0 < \beta < 1$,
\begin{align}
\frac{1}{1 - \eta^\beta} \le \frac{1}{1 - \eta} \cdot \frac{1}{\beta}.
\end{align}
\end{lem}

\begin{proof}[Proof of Lemma \ref{lem:dumb_bound}]

The desired inequality is equivalent to
\begin{align}
\beta \eta - \eta^\beta + 1  - \beta \ge 0.
\end{align}
Fix $0 < \beta < 1$, and denote the function on the left by $H(\eta)$.
Then $H(0) = 1 - \beta > 0$,
and $H(1) = \beta - 1 + 1 - \beta = 0$.
Furthermore,
\begin{align}
H'(\eta) = \beta - \beta \eta^{\beta-1}
= \beta(1 - \eta^{\beta-1}) < 0,
\end{align}
since $0 < \eta < 1$ and $\beta - 1 < 0$;
therefore, $H(\eta)$ decreases from $H(0) = 1- \beta$
to $H(1) = 0$, and hence must always be positive, as desired.

\end{proof}

\begin{proof}[Proof of Theorem \ref{thm:main_projections}]
Let
\begin{align}
M = \lfloor \log_{1/\omega} (\tau_0 / \varepsilon^2)\rfloor.
\end{align}

We break the proof into the three regimes of $\alpha$
appearing in the definition of $\Delta_{\alpha,p}$.

\item
\paragraph{Case 1: $0 < \alpha < 1$.}
If $k \le M$, then $\varepsilon / \sqrt{\tau_k} \le 1$,
and so, by Proposition \ref{prop:bound_layer},
\begin{align}
\| G_{\tau_k} \ast [\calP f - \calP f_\varphi]\|_{L^r}
\le C \cdot \|f\|_{M^r(\calU)} \cdot \frac{\varepsilon}{\sqrt{\tau_k}},
\end{align}
where $C$ depends only on the dimension $d$.
Therefore,
\begin{align}
\sum_{k=0}^{M} \tau_{k}^{p \alpha / 2} \| G_{\tau_k} \ast [\calP f - \calP f_\varphi]\|_{L^r}^p
&\le C^p \cdot \|f\|_{M^r(\calU)}^p \cdot
    \sum_{k=0}^{M} \tau_{k}^{p \alpha / 2} \frac{\varepsilon^p}{\tau_k^{p/2}}
\nonumber \\
&\le C^p \cdot \|f\|_{M^r(\calU)}^p \cdot \varepsilon^p \cdot \tau_0^{(\alpha-1)p/2}
    \sum_{k=0}^{M} \omega^{k(\alpha-1)p/2}
\nonumber \\
&\le C^p \cdot \|f\|_{M^r(\calU)}^p \cdot \varepsilon^p \cdot \tau_0^{(\alpha-1)p/2}
    \frac{\omega^{M(\alpha-1)p/2}}{1 - \omega^{(1-\alpha)p/2}}
\nonumber \\
&\le C^p \cdot \|f\|_{M^r(\calU)}^p \cdot \varepsilon^p \cdot
    \frac{\varepsilon^{(\alpha - 1)p}}{1 - \omega^{(1-\alpha)p/2}}
\nonumber \\
&= C^p \cdot \|f\|_{M^r(\calU)}^p \cdot \varepsilon^{p \alpha} \cdot
    \frac{1}{1 - \omega^{(1-\alpha)p/2}},
\end{align}
and so using Lemma \ref{lem:dumb_bound},
\begin{align}
\label{eq:bound_alpha_small-1}
\sum_{k=0}^{M} \tau_{k}^{p \alpha / 2} \| G_{\tau_k} \ast [\calP f - \calP f_\varphi]\|_{L^r}^p
\le C^p \cdot \|f\|_{M^r(\calU)}^p \cdot \varepsilon^{p \alpha} \cdot
    \frac{1}{1 - \omega^{p/2}} \cdot \frac{1}{1 - \alpha}.
\end{align}

For the remainder of the series,
we use the simpler bound
$\| G_{\tau_k} \ast [\calP f - \calP_\varphi]\|_{L^r} \le \| \calP f - \calP f_\varphi\|_{L^r}
\le 2 \| f \|_{M^r(\calU)}$,
and get
\begin{align}
\sum_{k=M+1}^{\infty} \tau_{k}^{p \alpha / 2} \| G_{\tau_k} \ast [\calP f - \calP f_\varphi]\|_{L^r}^p
&\le 2^p \| f \|_{M^r(\calU)}^p \sum_{k=M+1}^{\infty} \tau_{k}^{p \alpha / 2}
\nonumber \\
&\le 2^p \| f \|_{M^r(\calU)}^p \tau_0^{p \alpha/2}
        \sum_{k=M+1}^{\infty} \omega^{k p \alpha / 2}
\nonumber \\
&= 2^p \| f \|_{M^r(\calU)}^p \tau_0^{p \alpha/2}
        \frac{\omega^{(M+1)p \alpha / 2}}{1 - \omega^{p\alpha/2}}
\nonumber \\
&\le 2^p \| f \|_{M^r(\calU)}^p \cdot
        \frac{1}{1 - \omega^{p\alpha/2}} \cdot \varepsilon^{p \alpha}
\end{align}
and again
using Lemma \ref{lem:dumb_bound},
\begin{align}
\label{eq:bound_alpha_small-2}
\sum_{k=M+1}^{\infty} \tau_{k}^{p \alpha / 2} \| G_{\tau_k} \ast [\calP f - \calP f_\varphi]\|_{L^r}^p
\le 2^p \| f \|_{M^r(\calU)}^p \cdot \varepsilon^{p \alpha} \cdot
        \frac{1}{1 - \omega^{p/2}} \cdot \frac{1}{\alpha}.
\end{align}

Adding the bounds \eqref{eq:bound_alpha_small-1}
and \eqref{eq:bound_alpha_small-2} gives
\begin{align}
\D_{\alpha,p,r}(f,f_\varphi)^p
&\le C^p \cdot \|f\|_{M^r(\calU)}^p \cdot \varepsilon^{p \alpha} \cdot
    \frac{1}{1 - \omega^{p/2}} \cdot \frac{1}{1 - \alpha}
        + 2^p \| f \|_{M^r(\calU)}^p \cdot \varepsilon^{p \alpha} \cdot
        \frac{1}{1 - \omega^{p/2}} \cdot \frac{1}{\alpha}
\nonumber \\
&= C^p \cdot \|f\|_{M^r(\calU)}^p \cdot \frac{1}{1 - \omega^{p/2}} \cdot
    \left( \frac{1}{1 - \alpha} + \frac{1}{\alpha}\right) \cdot \varepsilon^{p \alpha}
\nonumber \\
&= C^p \cdot \|f\|_{M^r(\calU)}^p \cdot \frac{1}{1 - \omega^{p/2}} \cdot
    \frac{1}{\alpha(1 - \alpha)}  \cdot \varepsilon^{p \alpha},
\end{align}
and taking a $p$-th root and using Lemma \ref{lem:dumb_bound-2}
gives
\begin{align}
\D_{\alpha,p,r}(f,f_\varphi)
&\le C \cdot \|f\|_{M^r(\calU)} \cdot \frac{1}{1 - \omega^{1/2}} \cdot
    \frac{1}{(\alpha - \alpha^2)^{1/p}}  \cdot \varepsilon^{\alpha}
\nonumber \\
&= C \cdot \|f\|_{M^r(\calU)} \cdot \frac{\tau_0^{\alpha/2}}{1 - \omega^{1/2}} \cdot
    \frac{1}{(\alpha - \alpha^2)^{1/p}}  \cdot \varepsilon^{\alpha} / \tau_0^{\alpha/2}
\nonumber \\
&= C \cdot \|f\|_{M^r(\calU)} \cdot \frac{\tau_0^{\alpha/2}}{1 - \omega^{1/2}} \cdot
    \Delta_{\alpha,p}(\varepsilon / \sqrt{\tau_0}).
\end{align}

\item
\paragraph{Case 2: $\alpha = 1$.}

If $k \le M$, then $\varepsilon / \sqrt{\tau_k} \le 1$,
and so, by Proposition \ref{prop:bound_layer},
\begin{align}
\| G_{\tau_k} \ast [\calP f - \calP f_\varphi]\|_{L^r}
\le C \cdot \|f\|_{M^r(\calU)} \cdot \frac{\varepsilon}{\sqrt{\tau_k}},
\end{align}
where $C$ depends only on the dimension $d$.
Therefore,
\begin{align}
\label{eq:bound_alpha_one-1}
\sum_{k=0}^{M} \tau_{k}^{p  / 2} \| G_{\tau_k} \ast [\calP f - \calP f_\varphi]\|_{L^r}^p
&\le C^p \cdot \|f\|_{M^r(\calU)}^p \cdot
    \sum_{k=0}^{M} \tau_{k}^{p / 2} \frac{\varepsilon^p}{\tau_k^{p/2}}
\nonumber \\
&\le C^p \cdot \|f\|_{M^r(\calU)}^p \cdot \varepsilon^p \cdot M
\nonumber \\
&\le C^p \cdot \|f\|_{M^r(\calU)}^p \cdot \varepsilon^p
        \cdot \log_{1/\omega} (\tau_0 / \varepsilon^2).
\end{align}
For the remainder of the series,
we use the elementary bound
\begin{align}
\label{eq:log_dumb}
\log(1/t) \ge 1 - t, \quad 0 < t \le 1.
\end{align}
Indeed, the two sides agree at $t=1$ (where they are both $0$);
and the difference $\log(1/t) - 1 + t$ has
derivative $-1/t + 1 \le 0$.

Turning to the series, we again use the simpler bound
$\| G_{\tau_k} \ast [\calP f - \calP_\varphi]\|_{L^r} \le 2 \| f \|_{M^r(\calU)}$,
and get
\begin{align}
\label{eq:bound_alpha_one-2}
\sum_{k=M+1}^{\infty} \tau_{k}^{p / 2} \| G_{\tau_k} \ast [\calP f - \calP f_\varphi]\|_{L^r}^p
&\le 2^p \| f \|_{M^r(\calU)}^p \tau_0^{p /2}
        \sum_{k=M+1}^{\infty} \omega^{k p  / 2}
\nonumber \\
&= 2^p \| f \|_{M^r(\calU)}^p \tau_0^{p /2}
        \frac{\omega^{(M+1)p  / 2}}{1 - \omega^{p/2}}
\nonumber \\
&\le 2^p \| f \|_{M^r(\calU)}^p \cdot
        \frac{1}{1 - \omega^{p/2}} \cdot \varepsilon^{p }.
\end{align}

Adding the bounds \eqref{eq:bound_alpha_one-1}
and \eqref{eq:bound_alpha_one-2},
and using that $(p/2)\log(1/\omega) = \log(1/\omega^{p/2}) \ge 1 - \omega^{p/2}$,
gives
\begin{align}
\D_{1,p,r}(f,f_\varphi)^p
&\le C^p \cdot \|f\|_{M^r(\calU)}^p \cdot \varepsilon^p
        \cdot \log_{1/\omega} (\tau_0 / \varepsilon^2)
        + 2^p \| f \|_{M^r(\calU)}^p \cdot
        \frac{1}{1 - \omega^{p/2}} \cdot \varepsilon^{p }
\nonumber \\
&= C^p \cdot \|f\|_{M^r(\calU)}^p \cdot \frac{\varepsilon^p}{\tau_0^{p/2}}
        \cdot \log (\tau_0 / \varepsilon^2) \cdot \frac{\tau_0^{p/2}}{\log(1/\omega)}
        + 2^p \| f \|_{M^r(\calU)}^p \cdot
        \frac{\tau_0^{p/2}}{1 - \omega^{p/2}} \cdot \frac{\varepsilon^{p }}{\tau_0^{p/2}}
\nonumber \\
&= C^p \cdot \|f\|_{M^r(\calU)}^p \cdot \frac{\varepsilon^p}{\tau_0^{p/2}}
        \cdot \left(\log (\tau_0 / \varepsilon^2) \cdot \frac{\tau_0^{p/2}}{\log(1/\omega)}
            + \frac{\tau_0^{p/2}}{1 - \omega^{p/2}}\right)
\nonumber \\
&= C^p \cdot \|f\|_{M^r(\calU)}^p \cdot \frac{\varepsilon^p}{\tau_0^{p/2}}
        \cdot \left(\log (\tau_0 / \varepsilon^2) \cdot \frac{(p/2)\tau_0^{p/2}}{(p/2)\log(1/\omega)}
            + \frac{\tau_0^{p/2}}{1 - \omega^{p/2}}\right)
\nonumber \\
&= C^p \cdot \|f\|_{M^r(\calU)}^p \cdot \tau_0^p \cdot \frac{\varepsilon^p}{\tau_0^{p/2}}
        \cdot \left(\log (\sqrt{\tau_0} / \varepsilon) \cdot \frac{p}{1 - \omega^{p/2}}
            + \frac{1}{1 - \omega^{p/2}}\right)
\nonumber \\
&\le C^p \cdot \|f\|_{M^r(\calU)}^p \cdot \tau_0^{p/2} \cdot \frac{\varepsilon^p}{\tau_0^{p/2}}
        \cdot \left(\log (\sqrt{\tau_0} / \varepsilon)
            + 1 \right) \cdot \frac{p}{1 - \omega^{p/2}},
\end{align}
and taking a $p$-th root and using Lemma \ref{lem:dumb_bound-2}
and that $p^{1/p}$ is bounded gives
\begin{align}
\D_{1,p,r}(f,f_\varphi)^p
&\le C \cdot \|f\|_{M^r(\calU)} \cdot \frac{\tau_0^{1/2}}{1 - \omega^{1/2}}
        \cdot \frac{\varepsilon}{\tau_0^{1/2}}
        \cdot \left(\log (\sqrt{\tau_0} / \varepsilon)
            + 1 \right)^{1/p}
\nonumber \\
&= C \cdot \|f\|_{M^r(\calU)} \cdot \frac{\tau_0^{1/2}}{1 - \omega^{1/2}} \cdot
        \Delta_{1,p}(\varepsilon/\sqrt{\tau_0}).
\end{align}

\item
\paragraph{Case 3: $\alpha > 1$.}

If $k \le M$, then $\varepsilon / \sqrt{\tau_k} \le 1$,
and so, by Proposition \ref{prop:bound_layer},
\begin{align}
\| G_{\tau_k} \ast [\calP f - \calP f_\varphi]\|_{L^r}
\le C \cdot \|f\|_{M^r(\calU)} \cdot \frac{\varepsilon}{\sqrt{\tau_k}},
\end{align}
where $C$ only depends on the dimension $d$.
Therefore,
\begin{align}
\label{eq:bound_alpha_big-1}
\sum_{k=0}^{M} \tau_{k}^{p \alpha / 2} \| G_{\tau_k} \ast [\calP f - \calP f_\varphi]\|_{L^r}^p
&\le C^p \cdot \|f\|_{M^r(\calU)}^p \cdot
    \sum_{k=0}^{M} \tau_{k}^{p \alpha / 2} \frac{\varepsilon^p}{\tau_k^{p/2}}
\nonumber \\
&\le C^p \cdot \|f\|_{M^r(\calU)}^p \cdot \varepsilon^p \cdot \tau_0^{(\alpha-1)p/2}
    \sum_{k=0}^{M} \omega^{k(\alpha-1)p/2}
\nonumber \\
&\le C^p \cdot \|f\|_{M^r(\calU)}^p \cdot \varepsilon^p \cdot
    \frac{\tau_0^{(\alpha-1)p/2}}{1 - \omega^{(\alpha-1)p/2}}.
\end{align}

For the remainder of the series,
we again use the simpler bound
$\| G_{\tau_k} \ast [\calP f - \calP_\varphi]\|_{L^r} \le 2 \| f \|_{M^r(\calU)}$,
and, using that $(\varepsilon^2 / \tau_0)^\alpha \le \varepsilon^2 / \tau_0$,
we get
\begin{align}
\label{eq:bound_alpha_big-2}
\sum_{k=M+1}^{\infty} \tau_{k}^{p \alpha / 2} \| G_{\tau_k} \ast [\calP f - \calP f_\varphi]\|_{L^r}^p
&\le 2^p \| f \|_{M^r(\calU)}^p \tau_0^{p \alpha/2}
        \sum_{k=M+1}^{\infty} \omega^{k p \alpha / 2}
\nonumber \\
&= 2^p \| f \|_{M^r(\calU)}^p \tau_0^{p \alpha/2}
        \frac{\omega^{(M+1)p \alpha / 2}}{1 - \omega^{p\alpha/2}}
\nonumber \\
&\le 2^p \| f \|_{M^r(\calU)}^p \tau_0^{p \alpha/2}
        \frac{(\varepsilon^2 / \tau_0)^{p \alpha / 2}}{1 - \omega^{p\alpha/2}}
\nonumber \\
&\le 2^p \| f \|_{M^r(\calU)}^p \tau_0^{p \alpha/2}
        \frac{(\varepsilon^2 / \tau_0)^{p / 2}}{1 - \omega^{p\alpha/2}}
\nonumber \\
&\le 2^p \| f \|_{M^r(\calU)}^p \cdot \varepsilon^{p} \cdot
        \frac{\tau_0^{(\alpha-1)p/2}}{1 - \omega^{p\alpha/2}}.
\end{align}

Adding the bounds \eqref{eq:bound_alpha_big-1}
and \eqref{eq:bound_alpha_big-2},
\begin{align}
\D_{\alpha,p,r}(f,f_\varphi)^p
&\le C^p \cdot \|f\|_{M^r(\calU)}^p \cdot \varepsilon^p \cdot
    \frac{\tau_0^{(\alpha-1)p/2}}{1 - \omega^{(\alpha-1)p/2}}
        + 2^p \| f \|_{M^r(\calU)}^p \cdot \varepsilon^{p} \cdot
        \frac{\tau_0^{(\alpha-1)p/2}}{1 - \omega^{p\alpha/2}}
\nonumber \\
&= C^p \cdot \|f\|_{M^r(\calU)}^p \cdot \frac{\varepsilon^p}{\tau_0^{p/2}}
    \cdot \tau_0^{\alpha p/2} \cdot \left(\frac{1}{1 - \omega^{(\alpha-1)p/2}}
            + \frac{1}{1 - \omega^{p\alpha/2}}\right),
\end{align}
and taking $p$-th roots,
\begin{align}
\D_{\alpha,p,r}(f,f_\varphi)
&\le C \cdot \|f\|_{M^r(\calU)} \cdot \frac{\varepsilon}{\tau_0^{1/2}}
    \cdot \tau_0^{\alpha/2} \cdot \left(\frac{1}{1 - \omega^{(\alpha-1)p/2}}
            + \frac{1}{1 - \omega^{p\alpha/2}}\right)^{1/p}
\nonumber \\
&\le C \cdot \|f\|_{M^r(\calU)} \cdot \frac{\varepsilon}{\tau_0^{1/2}}
    \cdot \tau_0^{\alpha/2} \cdot \left[\left(\frac{1}{1 - \omega^{(\alpha-1)p/2}}\right)^{1/p}
            + \left(\frac{1}{1 - \omega^{p\alpha/2}} \right)^{1/p}\right]
\nonumber \\
&\le C \cdot \|f\|_{M^r(\calU)} \cdot \frac{\varepsilon}{\tau_0^{1/2}}
    \cdot \tau_0^{\alpha/2} \cdot \left[\left(\frac{1}{1 - \omega^{(\alpha-1)p/2}}\right)^{1/p}
            + \frac{1}{1 - \omega^{1/2}} \right].
\end{align}

For the last term, consider two cases.
When $1 < \alpha < 2$, we have the bound
\begin{align}
\frac{1}{1 - \omega^{(\alpha-1)p/2}}
\le \frac{1}{1 - \omega^{p/2}} \cdot \frac{1}{\alpha - 1},
\end{align}
and so
\begin{align}
\left(\frac{1}{1 - \omega^{(\alpha-1)p/2}}\right)^{1/p}
\le \frac{1}{1 - \omega^{1/2}} \cdot \frac{1}{(\alpha - 1)^{1/p}},
\end{align}
whereas for $\alpha \ge 2$, we have the simpler bound
\begin{align}
\left( \frac{1}{1 - \omega^{(\alpha-1)p/2}} \right)^{1/p}
\le \left(\frac{1}{1 - \omega^{p/2}} \right)^{1/p}
\le \frac{1}{1-\omega^{1/2}}.
\end{align}
Putting these cases together,
\begin{align}
\label{eq:54930201}
\left(\frac{1}{1 - \omega^{(\alpha-1)p/2}}\right)^{1/p}
&\le \frac{1}{1 - \omega^{1/2}} \cdot \frac{1}{\min\{1,(\alpha-1)^{1/p}\}}
\end{align}
and so the bracketed term may be bounded
\begin{align}
\left(\frac{1}{1 - \omega^{(\alpha-1)p/2}}\right)^{1/p}
            + \frac{1}{1 - \omega^{1/2}}
&\le \frac{1}{1 - \omega^{1/2}} \cdot \frac{2}{\min\{1,(\alpha-1)^{1/p}\}}
\end{align}
and consequently
\begin{align}
\D_{\alpha,p,r}(f,f_\varphi)
\le C \cdot \|f\|_{M^r(\calU)} \cdot \frac{\varepsilon}{\tau_0^{1/2}}
    \cdot  \frac{\tau_0^{\alpha/2}}{1 - \omega^{1/2}}\cdot \frac{1}{\min\{1,(\alpha-1)^{1/p}\}}
= C \cdot \|f\|_{M^r(\calU)}
    \cdot  \frac{\tau_0^{\alpha/2}}{1 - \omega^{1/2}}
        \cdot \Delta_{\alpha,p}(\varepsilon/\sqrt{\tau_0}),
\end{align}
as claimed.

\end{proof}

\subsection{Changes in projection angle}

Suppose $f : \RR^3 \to \RR$, $U$ is in SO(3), and $f_U(x) = f(Ux)$.
Let $e_z = (0,0,1)$, $v = U e_z$, and $\theta \in [0,\pi)$ be the angle
between $e_z$ and $v$.
Let $\calP$ denote the tomographic projection operator onto the $xy$-plane:
\begin{align}
(\calP f)(x,y) = \int f(x,y,z) \,dz.
\end{align}

We then have the following result:
\begin{thm}
\label{thm:rotations_projections}
Let $f$ be supported on the ball $B(0,R)$
of radius $R > 0$.
Suppose $\tau_0 \ge (2 R \sin(\theta/2))^2$.
Then
\begin{align}
\D_{\alpha,p,r}(\calP f, \calP f_U; \, \text{SO(2)})
\le C \cdot \|f\|_{M^r} \cdot \frac{\tau_0^{\alpha/2}}{1 - \omega^{1/2}}
    \cdot \Delta_{\alpha,p}(2 R \sin(\theta/2) / \sqrt{\tau_0}),
\end{align}
where $C>0$ is a universal constant.
\end{thm}

It is convenient to first state one elementary lemma,
whose proof we sketch for the reader's convenience.

\begin{lem}
\label{lem:max_dist_rot}
Suppose $U$ is in $\SO(3)$, and consider its
domain to be $B(0,R)$, the ball of radius $R>0$
centered at the origin. Then
\begin{align}
\varepsilon(U)^2 = \max_{x \in B(0,R)}|x - Ux|^2
= R^2(3 - \tr(U)).
\end{align}
\end{lem}

\begin{proof}[Proof of Lemma \ref{lem:max_dist_rot}]
Without loss of generality assume $R=1$.
Let $v$ be the fixed axis of $U$, so
that in an orthonormal basis starting with $v$, the matrix representation
of $U$ is of the form
\begin{align}
\left[
\begin{array}{ccc}
1          & 0           & 0               \\
0          & \cos(\phi)  & -\sin(\phi) \\
0          & \sin(\phi)  & \cos(\phi)  
\end{array}
\right],
\end{align}
where $\phi$ is the angle which $U$ rotates points
around $v$. Then $U$ moves all points on the unit circle
of the plane orthogonal to $v$ the same distance,
whose squared value is
\begin{align}
\varepsilon(U)^2 = (1 - \cos(\phi))^2 + \sin(\phi)^2
= 2(1 - \cos(\phi))
= 3 - \tr(U),
\end{align}
as claimed.

\end{proof}

\begin{proof}[Proof of Theorem \ref{thm:rotations_projections}]

Without loss of generality,
suppose that $v$ lies in the $yz$-plane. Then the rotation
$U = U(\theta,\psi)$ may be described by the composition of a rotation by $\theta$
in the $yz$ plane (bringing $v$ to $e_z$) with a rotation in the $xy$-plane
by another angle, $\psi$:
\begin{align}
U(\theta,\psi) &=
\left[
\begin{array}{ccc}
\cos(\psi)    & -\sin(\psi)    & 0        \\
\sin(\psi)    & \cos(\psi)     & 0        \\
0             & 0              & 1        
\end{array}
\right]
\left[
\begin{array}{ccc}
1          & 0           & 0               \\
0          & \cos(\theta)  & -\sin(\theta) \\
0          & \sin(\theta)  & \cos(\theta)  
\end{array}
\right]
\nonumber \\
&= 
\left[
\begin{array}{ccc}
\cos(\psi)          & -\sin(\psi) \cos(\theta)      & \sin(\psi)\sin(\theta)          \\
\sin(\psi)          & \cos(\psi) \cos(\theta)       & -\cos(\psi)\sin(\theta)          \\
0                   & \sin(\theta)                  & \cos(\theta)  
\end{array}
\right].
\end{align}

By Lemma \ref{lem:max_dist_rot},
\begin{align}
\varepsilon(U(\theta,\psi))^2
= R^2(3 - \tr(U(\theta,\psi)))
= R^2(3 - \cos(\psi) - \cos(\psi) \cos(\theta) - \cos(\theta)).
\end{align}
Minimizing this over the choice of in-plane rotation $\psi$
gives $\psi = 0$,
and the corresponding squared distance is
\begin{align}
\varepsilon(U(\theta,0))^2
= R^2 (2 - 2 \cos(\theta))
= R^2 (2 - 2 (1 - 2\sin(\theta/2)^2))
= 4 R^2 \sin(\theta/2)^2.
\end{align}

Consequently, since $\psi$
determines the in-plane rotation,
\begin{align}
\D_{\alpha,p,r}(\calP f, \calP f_{U(\theta,\psi)}; \, \text{SO(2)})
&= \min_{Q \in \SO(2)} \D_{\alpha,p,r}(\calP f, \calP f_{U(\theta,\psi)} \circ Q)
\nonumber \\
&= \min_{Q \in \SO(2)} \D_{\alpha,p,r}(\calP f, \calP f_{U(\theta,0)} \circ Q)
\nonumber \\
&\le \D_{\alpha,p,r}(\calP f, \calP f_{U(\theta,0)}).
\end{align}
By Theorem \ref{thm:main_projections}, then, we have the bound
\begin{align}
\D_{\alpha,p,r}(\calP f, \calP f_{U(\theta,0)})
\le C \cdot \|f\|_{M^r} \cdot \frac{\tau_0^{\alpha/2}}{1 - \omega^{1/2}}
    \cdot \Delta_{\alpha,p}(2 R \sin(\theta/2) / \sqrt{\tau_0}),
\end{align}
as claimed.

\end{proof}

\begin{rmk}
Similar results bounding the distance between
tomographic projections by
an increasing function of the angle between projection directions
are known for rotationally-invariant
Wasserstein distance \cite{rao2020wasserstein}
and rotationally-invariant sliced Wasserstein distance \cite{shi2025fast}.
Using the results from \cite{leeb2025sliced},
a similar result can be easily deduced for the sliced Cram\'er metrics
as well.
\end{rmk}

\subsection{Convolutions}

Images observed in scientific problems
are often convolved with a linear filter
induced from the measurement apparatus.
In cryo-EM, for example, the projection images
are filtered by the contrast transfer function \cite{bendory2020single, singer2020computational, van2000single, doerr2016single, vulovic2013image}.
It is therefore of interest to consider the effect
of such linear filters on distances between the images.
The following result is immediate:

\begin{prop}
\label{prop:convolutions}
Let $\alpha > 0$, $p \ge 1$, $r \ge 1$.
Suppose $f$ and $g$ are in
$L^r$,
and $w$ is in $L^1$.
Then
\begin{align}
\label{eq:convolutions}
\D_{\alpha,p,r}(f \ast w, g \ast w) \le
\|w\|_{L^1} \cdot \D_{\alpha,p,r}(f, g).
\end{align}

\end{prop}

\begin{proof}
The proof is an immediate consequence of Young's inequality
for convolutions.
Indeed, for any $\tau$,
\begin{align}
\|G_{\tau} \ast ((f \ast w) - (g \ast w))\|_{L^r}
= \|w \ast G_{\tau} \ast (f - g)\|_{L^r}
\le \|w\|_{L^1} \|G_{\tau} \ast (f - g)\|_{L^r},
\end{align}
from which the result follows.
\end{proof}

\begin{rmk}
Of course, the analogous result holds for
any $L^r$ distance, also due to Young's inequality.
It is also known for other metrics,
such as the sliced Cram\'er metric \cite{zhang2023cramer, leeb2025sliced}.
\end{rmk}

\section{Discretizations in 2D}

We describe Fourier-based discretizations
of the metrics in 2D.
We only consider the case where $r=2$,
and denote the corresponding metric $\D_{\alpha,p}(g,h)$.

Throughout, we let $g$ and $h$ be supported in $B(0,R_0)$,
the ball of radius $R_0 > 0$ centered at the origin;
and we let $L_0 = 2 R_0$.
We suppose $R > R_0$ is big enough so that $|(G_{\tau_0} \ast g )(t)|$
and $|(G_{\tau_0} \ast h )(t)|$
are negligibly small for $|t| \ge R$.
More precisely,
let $R = R_0 + \sqrt{\tau_0 \cdot  \log(1/\epsilon)}$,
where $\epsilon$ is a specified tolerance.
We also let $L = 2R$.
Let $t_j = -R + 2Rj / n$, $j=0,\dots,n-1$,
For simplicity, we will assume that $n$ is even.
We assume that we observe samples
$g(t_{j_1}, t_{j_2})$ and $h(t_{j_1}, t_{j_2})$
of $g$ and $h$ on the Cartesian grid;
note that if we only observe samples from
the grid over $(-R_0,R_0) \times (-R_0,R_0)$,
the additional samples over the full
domain $(-R,R) \times (-R,R)$ can be obtained from zero-padding.

\subsection{Discretization in 2D, without rotations}

In this section, we let $f = g-h$;
our goal is fast approximation of $\D_{\alpha,p}(g,h) = \|f\|_{\Gamma_p^\alpha}$.
For $j_1,j_2 = 0,\dots,n-1$, let $x[j_1,j_2] = f(t_{j_1}, t_{j_2}) = g(t_{j_1}, t_{j_2})
- h(t_{j_1}, t_{j_2})$
denote the samples of $f$ on the Cartesian grid.
We define the normalized DFT of $x$ by
\begin{align}
\what{x}[k_1,k_2]
    = \frac{L^2}{n^2} \sum_{j_1=0}^{n-1} \sum_{j_2=0}^{n-1}x[j_1,j_2]
        e^{-2 \pi \i (k_1 t_{j_1}  + k_2 t_{j_2} )/L},
\end{align}
so that, for $\max\{|k_1|, |k_2|\} \le n/2$,
\begin{align}
\what{x}[k_1,k_2] \approx \what{f}(k_1/L,k_2/L).
\end{align}

Denote by $\what{G}_{\tau}(\xi_1,\xi_2)$
the Gaussian Fourier transform:
\begin{align}
\what{G}_{\tau}(\xi_1,\xi_2) = e^{- \pi^2 \tau (\xi_1^2 + \xi_2^2)},
\end{align}
and denote the approximate Fourier coefficients of $G_\tau \ast f $
by
\begin{align}
\what{x}_\tau[k_1,k_2] = \what{x}[k_1,k_2] \what{G}_{\tau}(k_1/L,k_2/L),
\end{align}
so that 
\begin{align}
\what{x}_\tau[k_1,k_2] \approx \what{(G_\tau \ast f)}(k_1/L,k_2/L).
\end{align}

We then define the approximation to the $L^2$ norm
at level $\tau$ by
\begin{align}
S(\tau)^2 = \frac{1}{L^2}\sum_{k_1=-n/2}^{n/2-1} \sum_{k_2=-n/2}^{n/2-1}
        |\what{x}_{\tau}[k_1,k_2]|^2.
\end{align}
Then
\begin{align}
S(\tau)^2 \approx \|G_{\tau} \ast f\|_{L^2}^2.
\end{align}
Finally, we define the approximation to $\|f\|_{\Gamma_p^{\alpha}}$
by
\begin{align}
\|x\|_{\gamma_p^\alpha}
    = \left(\sum_{\ell = 0}^{M_n} \tau_\ell^{p \alpha / 2} S(\tau_\ell)^p \right)^{1/p},
\end{align}
for a specified $M_n = O(\log(n))$.
Then $\|x\|_{\gamma_p^\alpha} \approx \|f\|_{\Gamma_p^\alpha}
= \D_{\alpha,p}(g,h)$.

\begin{rmk}
Note that the calculation of $\|x\|_{\gamma_p^\alpha}$ can be performed
with $O(n^2 \log(n))$ floating-point operations. Indeed, a 2D FFT
computes all $\what{x}[k_1,k_2]$ at cost $O(n^2 \log(n))$;
and at each of the $M_n = O(\log(n))$ layers, $n^2$ multiplications
are performed to evaluate all $\what{x}_{\tau_\ell}[k_1,k_2]$
and evaluate the sums defining $S(\tau_\ell)^2$;
then an additional $M_n = O(\log(n))$ operations are required to
sum over all $M_n$ levels.
\end{rmk}

We prove an error bound:
\begin{thm}
\label{thm:error_no_rotations}
Suppose $g$ and $h$ are $C^{\infty}(\R)$,
and supported in $B(0,R_0)$.
Let $f = g-h$. Fix $J > 1$,
let $0 < \epsilon < n^{-J}$,
and $R = R_0 + \sqrt{\tau_0 \cdot  \log(1/\epsilon)}$.
Then using $M_n = \lfloor (2J/p\alpha) \log_{1/\omega}(n) \rfloor$
levels in the definition of $\|x\|_{\gamma_p^\alpha}$,
\begin{align}
\left| \|x\|_{\gamma_p^\alpha} - \|f\|_{\Gamma_p^\alpha}\right| = O\left(n^{-J} \right).
\end{align}
\end{thm}

\begin{rmk}
The constants implicit in the notation ``$O\left(n^{-J} \right)$''
may depend on $J$, $f$, $\alpha$, $\omega$, $\tau_0$, and $p$.
\end{rmk}

\begin{proof}[Proof of Theorem \ref{thm:error_no_rotations}]

It will first be convenient to record several elementary
lemmas, some of whose proofs we provide for the reader's convenience.

\begin{lem}
\label{lem:mean_value_theorem}
Suppose
and $|a - \wtilde{a}| \le \delta$
for some $\delta > 0$. Let $q \ge 1$.
Then $\left| |a|^q - |\wtilde{a}|^q \right| \le q \cdot (|a| + \delta)^{q-1} \cdot \delta$.
\end{lem}

\begin{proof}[Proof of Lemma \ref{lem:mean_value_theorem}]
Let $h(x) = |x|^q$, so that $h'(x) = q |x|^{q-1} \sign(x)$.
By the Mean Value Theorem,
there is some $b$ between $a$ and $\wtilde a$
such that 
\begin{align}
||a|^q - |\wtilde{a}|^q| = |h(a) - h(\wtilde{a})| = |h'(b)| \cdot |a - \wtilde{a}|
= q |b|^{q-1} \cdot \delta
< q \cdot (|a| + \delta)^{q-1} \cdot \delta,
\end{align}
where we have used that $|b| \le \max\{|a|,|\wtilde{a}|\} \le |a| + \delta$.

\end{proof}

\begin{cor}
\label{cor:mvt_complex}
Suppose $z$ and $\wtilde{z}$ are complex,
with $|z - \wtilde{z}| \le \delta$ for some
$\delta > 0$.
Then $\left| |z|^2 - |\wtilde{z}|^2 \right| \le 2 \sqrt{2} \cdot |z|  \cdot \delta + 4 \delta^2$.
\end{cor}

\begin{proof}[Proof of Corollary \ref{cor:mvt_complex}]
If $z = a+\i b$ and $\wtilde z = \wtilde a + \i \wtilde b$,
then from Lemma \ref{lem:mean_value_theorem},
since $|a - \wtilde{a}| \le |z - \wtilde{z}| \le \delta$, therefore
$\left| a^2 - \wtilde{a}^2 \right| \le 2 \cdot (|a|+\delta) \cdot \delta$
and similarly,
$\left| b^2 - \wtilde{b}^2 \right| \le 2 \cdot (|b| + \delta) \cdot \delta$,
and so
\begin{align}
\left| |z|^2 - |\wtilde z|^2\right|
= \left| a^2 + b^2 - (\wtilde a^2 + \wtilde b^2) \right|
\le \left| a^2 -\wtilde a^2 \right| + \left| b^2 -\wtilde b^2 \right|
\le 2 \cdot (|a|+\delta) \cdot \delta + 2 \cdot (|b|+\delta) \cdot \delta
\le 2 \sqrt{2} \cdot |z|  \cdot \delta + 4 \delta^2,
\end{align}
as claimed.

\end{proof}

\begin{lem}
\label{lem:concavity}
Suppose $a,\wtilde{a} > 0$,
and $|a - \wtilde{a}| \le \delta$
for some $\delta > 0$. Let $p \ge 1$.
Then $\left| a^{1/p} - \wtilde{a}^{1/p} \right| \le \delta^{1/p}$.
\end{lem}

\begin{proof}[Proof of Lemma \ref{lem:concavity}]
This is immediate from the inequality $(a + b)^{1/p} \le a^{1/p} + b^{1/p}$
for $a,b > 0$. Indeed, we have
$a \le \wtilde{a} + \delta$, and so
$a^{1/p} \le (\wtilde{a} + \delta)^{1/p} \le \wtilde{a}^{1/p} + \delta^{1/p}$,
hence $a^{1/p} - \wtilde{a}^{1/p} \le \delta^{1/p}$;
and similarly $\wtilde{a}^{1/p} - a^{1/p} \le \delta^{1/p}$,
completing the proof.

\end{proof}

\begin{lem}
\label{lem:concavity_v2}
Suppose $a,\wtilde{a} > 0$, and
$|a - \wtilde{a}| < \delta$ for $0 < \delta < a/2$.
Let $p \ge 1$. Then
\begin{align}
\left| a^{1/p} - \wtilde{a}^{1/p} \right| \le (1/p) (a/2)^{1/p-1}  \delta.
\end{align}
\end{lem}

\begin{proof}
First, note that, since $a - \wtilde{a} < \delta < a/2$,
we have $\wtilde{a} > a/2$.
Let $q(t) = t^{1/p}$, for $t \ge 0$.
Then $q'(t) = (1/p) t^{1/p-1}$, which is decreasing.
By the Mean Value Theorem, there exists $b$ between $a$
and $\wtilde{a}$ such that
\begin{align}
\left| a^{1/p} - \wtilde{a}^{1/p}\right| \le (1/p) b^{1/p-1} |a - \wtilde{a}|
\le (1/p) (a/2)^{1/p-1} \delta,
\end{align}
where we have used $b \ge a/2$.

\end{proof}

\begin{lem}
\label{lem:log_geometric}
The function
\begin{align}
q(\delta) = \frac{\delta}{1 - \delta} \log(1 / \delta)
\end{align}
is increasing on $(0,1)$, converges to $0$ as $\delta \to 0^+$,
and converges to $1$ as $\delta \to 1^-$.
\end{lem}

\begin{proof}[Proof of Lemma \ref{lem:log_geometric}]
The limits are obtained by l'H\^opital's rule. The derivative
of the function is
\begin{align}
q'(\delta) = \frac{\delta-1-\log(\delta)}{(\delta-1)^2},
\end{align}
which is positive over $0 < \delta < 1$
since $\log(\delta) < \delta - 1$.

\end{proof}

\begin{lem}
\label{lem:gaussian_crossover}
If $\tau < \tau_0$, then $G_\tau(x) \le G_{\tau_0}(x)$
whenever
\begin{align}
|x| \ge \sqrt{\tau_0 \tau\frac{\log(\tau_0 / \tau)}{\tau_0 - \tau}}.
\end{align}
Furthermore, the right side decreases as  $\tau$ decreases.
\end{lem}

\begin{proof}[Proof of Lemma \ref{lem:gaussian_crossover}]
The first part is immediate from rearranging the inequality $G_{\tau_0}(x) \ge G_{\tau}(x)$.
The second part follows by rewriting the square of the right side as
\begin{align}
\tau_0 \tau\frac{\log(\tau_0 / \tau)}{\tau_0 - \tau}
= \tau_0 \frac{\tau/\tau_0}{1 - \tau/\tau_0} \log(1 / (\tau/\tau_0))
\end{align}
and applying Lemma \ref{lem:log_geometric} with $\delta = \tau/\tau_0$.

\end{proof}

We break the error analysis into several steps.

\item
\paragraph{Step 1.}
For $\tau < \tau_0$,
we will first bound the error
\begin{align}
\left| \frac{1}{L^2}\sum_{(k_1,k_2) \in \ZZ^2}
        |\what{f}(k_1/L,k_2/L)|^2 \what{G}_\tau(k_1/L,k_2/L)^2
- \|G_\tau \ast f\|_{L^2}^2 \right|.
\end{align}

Let $f_s(x) = f(-x)$ for $x \in \R^2$,
and let $H_\tau = (G_{\tau} \ast f) \ast (G_{\tau} \ast f_s)$.
Note that
\begin{align}
H_{\tau}(0,0) = \| G_\tau \ast f \|_{L^2}^2.
\end{align}
Also note that $H_\tau$ has Fourier transform
\begin{align}
\what{H}_\tau(\xi_1,\xi_2) = |\what{f}(\xi_1,\xi_2)|^2 \what{G}_\tau(\xi_1,\xi_2)^2.
\end{align}

By the Poisson summation formula (e.g.\ Theorem 8.32 in \cite{folland1999real}),
\begin{align}
\frac{1}{L^2} \sum_{k_1 = -\infty}^{\infty} \sum_{k_2 = -\infty}^{\infty}
    |\what{f}(k_1/L,k_2/L)|^2 \what{G}_\tau(k_1/L,k_2/L)^2
&=\frac{1}{L^2} \sum_{k_1 = -\infty}^{\infty} \sum_{k_2 = -\infty}^{\infty}
    \what{H}_\tau(k_1/L,k_2/L)
\nonumber \\
&= \sum_{j_1 = -\infty}^{\infty} \sum_{j_2 = -\infty}^{\infty}H_\tau(j_1 L,j_2 L)
\nonumber \\
&= H_\tau(0,0) + \sum_{(j_1,j_2) \ne (0,0)} H_\tau(j_1 L,j_2 L).
\nonumber \\
&= \|G_\tau \ast f\|_{L^2}^2 + \sum_{(j_1,j_2) \ne (0,0)} H_\tau(j_1 L,j_2 L).
\end{align}

We now need to bound the remainder term.
Let $F = f \ast f_s$, and note that, since $f$
is assumed to be supported in $B(0,R_0)$, $F$
is supported in $B(0,L_0)$, where $L_0 = 2R_0$.
Note too that $G_\tau \ast G_{\tau} = G_{2 \tau}$,
and so
\begin{align}
H_\tau = F \ast G_{2 \tau}.
\end{align}
Therefore,
\begin{align}
\sum_{j \in \ZZ^2 \setminus (0,0)} H_\tau(Lj)
&= \sum_{j \in \ZZ^2 \setminus (0,0)} \int_{B(0,L_0)} F(x) G_{2 \tau}(x - jL) \,dx
\nonumber \\
&\le \int_{B(0,L_0)} |F(x)|  \sum_{j \in \ZZ^2 \setminus (0,0)} G_{2 \tau}(x - jL) \,dx.
\end{align}
We bound the series appearing in the integral uniformly, as follows.
First, since the $\ell_\infty$ ball of radius $L_0$ is
bigger than the $\ell_2$ ball of radius $L_0$,
it is enough to bound
\begin{align}
\sup_{|x|_\infty \le L_0} \sum_{j \in \ZZ^2 \setminus (0,0)} G_{2 \tau}(x - jL).
\end{align}
Suppose $j = (j_1 , j_2) \ne (0,0)$, and without loss of
generality suppose that $j_1 > 0$.
Then the maximum possible value of
$G_{2 \tau}(x - jL)$
for $|x|_\infty \le L_0$
is achieved when $x_1 = L_0$,
in which case
\begin{align}
|jL - x|^2
= |j_1 L - x_1|^2 + |j_2 L - x_2|^2
\ge |j_1 L - L_0|^2.
\end{align}

Now, for $\tau \le \tau_0$, say $\tau = \tau_0 \cdot \delta$
with $0 <\delta < 1$,
by Lemma \ref{lem:gaussian_crossover},
then
$G_{2\tau}(y) \le G_{2 \tau_0}(y)$
when
\begin{align}
|y|^2 &\ge 2\tau_0 \tau\frac{\log(\tau_0 / \tau)}{\tau_0 - \tau}
\nonumber \\
&= 2\tau_0^2 \delta
    \frac{\log(1/\delta)}{\tau_0 - \tau_0 \delta}
\nonumber \\
&= 2\tau_0 \frac{\delta}{1 - \delta}
    \log(1/\delta).
\end{align}
From Lemma \ref{lem:log_geometric},
this last quantity
is bounded above by $2\tau_0$;
hence, so long as $|y| \ge \sqrt{2\tau_0}$,
we have $G_{2\tau}(y) \le G_{2 \tau_0}(y)$.
With $x_1 = L_0$, we have
\begin{align}
|jL - x|^2
\ge |j_1 L - L_0|^2
\ge |L - L_0|^2
= 4 |R - R_0|^2
= 4 \tau_0 \cdot  \log(1/\epsilon)
\ge 4 \tau_0,
\end{align}
it then follows that
for this choice of $x$, $G_{2\tau}(|jL - x|) \le G_{2 \tau_0}(|jL - x|)$.
Hence, it is enough to prove an upper bound for $\tau = \tau_0$.

To that end, we have
\begin{align}
G_{2 \tau_0}(x - jL)
&\le \frac{1}{2 \tau_0}
        e^{-(j_1 L - x_1)^2 / (2 \tau_0)}
\nonumber \\
&\le \frac{1}{2 \tau_0}
        e^{-(j_1 L - L_0)^2 / (2 \tau_0)}
\nonumber \\
&= \frac{1}{2 \tau_0}
        e^{-(j_1 L - j_1 L_0 + (j_1 - 1)L_0)^2 / (2 \tau_0)}
\nonumber \\
&\le \frac{1}{2 \tau_0}
        e^{-(j_1 L - j_1 L_0)^2 / (2 \tau_0)}
\nonumber \\
&= \frac{1}{2 \tau_0}
        e^{-j_1^2 (L - L_0)^2 / (2 \tau_0)}
\nonumber \\
&= \frac{1}{2 \tau_0}
        e^{-j_1^2 4 \tau_0 \log(1/\epsilon) / (2 \tau_0)}
\nonumber \\
&= \frac{1}{2 \tau_0}  \epsilon^{2 j_1^2}.
\end{align}
The same holds, of course, if $j_1 < 0$.
Similarly, if $j_2 \ne 0$,
then
\begin{align}
G_{2 \tau_0}(x - jL)
\le \frac{1}{2 \tau_0}  \epsilon^{2 j_2^2}.
\end{align}
And if both $j_1 \ne 0$ and $j_2 \ne 0$,
then similar reasoning shows that
\begin{align}
G_{2 \tau_0}(x - jL) \le \frac{1}{2 \tau_0} \epsilon^{2 (j_1^2 + j_2^2)}.
\end{align}
Hence
\begin{align}
\sum_{j \in \ZZ^2 \setminus (0,0)} G_{2 \tau}(x - jL)
&\le \frac{1}{2 \tau_0} \sum_{(j_1,j_2) \ne (0,0)} \epsilon^{2 (j_1^2 + j_2^2)}
\nonumber \\
&= \frac{1}{2 \tau_0} \left(
    \sum_{j_1 \ne 0} \epsilon^{2 j_1^2}
        +     \sum_{j_2 \ne 0} \epsilon^{2 j_2^2}
        +    \sum_{j_1 \ne 0,j_2 \ne 0} \epsilon^{2 j_1^2} \epsilon^{2 j_2^2}
\right)
\nonumber \\
&= \frac{1}{\tau_0}\sum_{m \ne 0} \epsilon^{2 m^2}
        + \frac{1}{2\tau_0}\left(\sum_{m \ne 0} \epsilon^{2 m^2} \right)^2
\nonumber \\
&\le \frac{C}{\tau_0} \epsilon^2,
\end{align}
for a universal constant $C>0$.

Therefore,
\begin{align}
\sum_{j \in \ZZ^2 \setminus (0,0)} H_\tau(Lj)
&\le \int_{B(0,L_0)} |F(x)|  \sum_{j \in \ZZ^2 \setminus (0,0)} G_{2 \tau}(x - jL) \,dx
\nonumber \\
&\le \frac{C}{\tau_0} \|F\|_{L_1} \epsilon^2
\nonumber \\
&\le \frac{C}{\tau_0} \|f\|_{L_1}^2 \epsilon^2,
\end{align}
where in the last inequality we have used that
$\|F\|_{L^1} = \|f \ast f_s\|_{L^1} \le \|f\|_{L^1} \|f_s\|_{L^1} = \|f\|_{L^1}^2$.
We have therefore shown that
\begin{align}
\left| \frac{1}{L^2} \sum_{k_1 = -\infty}^{\infty} \sum_{k_2 = -\infty}^{\infty}
    |\what{f}(k_1/L,k_2/L)|^2 \what{G}_\tau(k_1/L,k_2/L)^2
    - \|G_\tau \ast f\|_{L^2}^2 \right| \le \frac{C}{\tau_0} \|f\|_{L_1}^2 \epsilon^2.
\end{align}

\item
\paragraph{Step 2.}
We now bound the error from truncating the infinite sum.
Because $f$ is $C^\infty$, for any $J>0$, $|\what{f}(\xi)| = O(|\xi|^{-(J+2)})$,
and so $|\what{f}(k/L)|^2 = O(L^{2J+2}/|k|^{2J+2})$.
Furthermore, $\|\what{G}_{\tau}\|_{\infty} \le 1$ for all $\tau$.
Note too that for any $m \ge 1$, there are $8m$ $k \in \ZZ^2$
with $|k|_\infty = m$.
It then follows that the remainder term
can be bounded:
\begin{align}
&\left| \frac{1}{L^2}\sum_{k \in \ZZ^2}
            |\what{f}(k_1/L,k_2/L)|^2 \what{G}_\tau(k_1/L,k_2/L)^2
    -\frac{1}{L^2}\sum_{|k|_\infty \le n/2}
            |\what{f}(k_1/L,k_2/L)|^2 \what{G}_\tau(k_1/L,k_2/L)^2 \right|
\nonumber \\
= \, & \frac{1}{L^2}\sum_{|k|_\infty > n/2} |\what{f}(k_1/L,k_2/L)|^2 \what{G}_\tau(k_1/L,k_2/L)^2.
\nonumber \\
\le \, & C \cdot L^{2J}\sum_{|k|_\infty > n/2} \frac{1}{|k|^{2J+2}}
\nonumber \\
= \, & C \cdot L^{2J}\sum_{m=n/2+1}^{\infty}
        \sum_{|k|_\infty = m} \frac{1}{m^{2J+2}}
\nonumber \\
= \, & C \cdot L^{2J}\sum_{m=n/2+1}^{\infty} \frac{8}{m^{2J+1}}
\nonumber \\
\le \, & C \cdot L^{2J} \int_{n/2}^{\infty} \frac{1}{r^{2J+1}} \,dr
\nonumber \\
= \, & C \cdot  \frac{L^{2J}}{2J} \cdot \frac{1}{(n/2)^{2J}}
\nonumber \\
= \, & C \cdot  \frac{(2L)^{2J} }{J} \cdot \frac{1}{n^{2J}},
\end{align}
i.e. we have shown the bound
\begin{align}
\left| \frac{1}{L^2}\sum_{k \in \ZZ^2}
            |\what{f}(k_1/L,k_2/L)|^2 \what{G}_\tau(k_1/L,k_2/L)^2
    -\frac{1}{L^2}\sum_{|k|_\infty \le n/2}
            |\what{f}(k_1/L,k_2/L)|^2 \what{G}_\tau(k_1/L,k_2/L)^2 \right|
= O\left( n^{-2J} \right).
\end{align}

\item
\paragraph{Step 3.}
Finally, we account for the error
in approximating each $\what{f}(k_1/L,k_2/L)$
by $\what{x}[k_1,k_2]$.
Because $f$ is $C^\infty$,
a standard error bound (see e.g.\ \cite{leeb2025sliced}) says that
for $\max\{|k_1|,|k_2| \} \le n / (2L)$,
\begin{align}
|\what{x}[k_1,k_2] - \what{f}(k_1/L,k_2/L)| = O(n^{-(2J+2)}),
\end{align}
and so from Corollary \ref{cor:mvt_complex},
\begin{align}
\left| |\what{x}[k_1,k_2]|^2 - |\what{f}(k_1/L,k_2/L)|^2 \right|
    = O \left( n^{-(2J+2)} \right).
\end{align}
Consequently, since $\what{G}_\tau(k_1/L,k_2/L) \le 1$,
summing the errors over all $n^2$ terms gives the error bound
\begin{align}
\left| \sum_{k_1 = -n/2}^{n/2-1} \sum_{k_2 = -n/2}^{n/2-1}
    \left(|\what{x}[k_1,k_2]|^2 \what{G}_\tau(k_1/L,k_2/L)^2
            - |\what{f}(k_1/L,k_2/L)|^2 \what{G}_\tau(k_1/L,k_2/L)^2 \right) \right|
= O\left( n^{-2J} \right).
\end{align}

\item
\paragraph{Step 4.}
Putting all the bounds together,
\begin{align}
\left| \|G_\tau \ast f\|_{L^2}^2
         - \frac{1}{L^2}\sum_{k_1 = -n/2}^{n/2-1} \sum_{k_2 = -n/2}^{n/2-1}
                |\what{x}[k_1,k_2]|^2 \what{G}_\tau(k_1/L,k_2/L)^2 \right|
= O(n^{-2J} + \epsilon^2)
= O(n^{-2J}).
\end{align}

Taking square roots and employing Lemma \ref{lem:concavity},
\begin{align}
\left| \|G_\tau \ast f\|_{L^2}
         - \left(\frac{1}{L^2}\sum_{k_1 = -n/2}^{n/2-1} \sum_{k_2 = -n/2}^{n/2-1}
                |\what{x}[k_1,k_2]|^2 \what{G}_\tau(k_1/L,k_2/L)^2\right)^{1/2} \right|
= O(n^{-J}),
\end{align}
i.e.
\begin{align}
\left| \|G_\tau \ast f\|_{L^2} - S(\tau_\ell)\right| = O\left( n^{-J} \right).
\end{align}

Letting $\delta$ denote the $O(n^{-J})$ error,
for $n$ sufficiently big, $\delta \le \|f\|_{L^2}$, and so
by Lemma \ref{lem:mean_value_theorem},
\begin{align}
\left| \|G_\tau \ast f\|_{L^2}^p - S(\tau)^p \right|
\le C \cdot p \cdot (\|G_\tau \ast f\|_{L^2} + \delta)^{p-1} \cdot \delta
\le C \cdot p \cdot (2\|f\|_{L^2})^{p-1} \cdot n^{-J}
= O\left( n^{-J} \right).
\end{align}
Importantly, the constant hidden in the ``$O(n^{-J})$''
does not depend on $\tau$.

\item
\paragraph{Step 5.}
We now bound the error for the entire series defining the metric.
As we have shown, there is a $C>0$ such that for all $\tau$,
\begin{align}
\left| \|G_\tau \ast f\|_{L^2}^p - S(\tau)^p \right|
\le C \cdot n^{-J}.
\end{align}

Recall that $M_n = \lfloor (2J/p\alpha) \log_{1/\omega}(n) \rfloor$.
Therefore,
\begin{align}
\left| \sum_{k=0}^{M_n} \tau_k^{p \alpha / 2}
        (S(\tau_\ell)^p - \|G_{\tau_\ell} \ast f\|_{L^2}^p) \right|
&\le C \cdot n^{-J} \cdot \tau_0^{p \alpha/2} \sum_{k=0}^{M_n} \omega^{k p \alpha / 2}
\nonumber \\
&\le C \cdot n^{-J} \cdot \tau_0^{p \alpha/2} \frac{1}{1 - \omega^{p \alpha / 2}}.
\end{align}

Regarding the remainder,
\begin{align}
\sum_{k=M_n+1}^{\infty} \tau_k^{p \alpha / 2} \|G_{\tau_\ell} \ast f\|_{L^2}^p
&\le \|f\|_{L^2}^p \sum_{k=M_n+1}^{\infty} \tau_k^{p \alpha / 2}
\nonumber \\
&\le \|f\|_{L^2}^p \tau_0^{p\alpha/2}\sum_{k=M_n+1}^{\infty} \omega^{k p \alpha / 2}
\nonumber \\
&= \|f\|_{L^2}^p \tau_0^{p\alpha/2} \frac{\omega^{(M_n+1) p \alpha / 2}}{1 - \omega^{p \alpha / 2}}.
\end{align}

So, since $M_n = \lfloor (2J/p\alpha) \log_{1/\omega}(n) \rfloor$,
therefore
\begin{align}
\omega^{M_n+1} \le n^{-2J/p\alpha},
\end{align}
and the bound becomes
\begin{align}
\sum_{k=M_n+1}^{\infty} \tau_k^{p \alpha / 2} \|G_{\tau_\ell} \ast f\|_{L^2}^p
\le \|f\|_{L^2}^p \tau_0^{p\alpha/2} \frac{n^{-J}}{1 - \omega^{p \alpha / 2}}
\end{align}

Putting the bounds together,
\begin{align}
\left| \|x\|_{\gamma_p^\alpha}^p - \|f\|_{\Gamma_p^\alpha}^p \right|
= O\left( n^{-J}\right).
\end{align}
Assume $\|f\|_{\Gamma_p^\alpha} > 0$
(otherwise the result is trivial,
since $f \equiv 0$ and hence $x \equiv 0$). For $n$ sufficiently big,
$\|x\|_{\gamma_p^\alpha}^p \ge \|f\|_{\Gamma_p^\alpha}^p / 2$;
consequently,
by Lemma \ref{lem:concavity_v2},
\begin{align}
\left| \|x\|_{\gamma_p^\alpha} - \|f\|_{\Gamma_p^\alpha} \right|
= O\left( n^{-J}\right),
\end{align}
concluding the proof.

\end{proof}

\subsection{Discretization in 2D, with in-plane rotations}
\label{sec:discretization_rotations}

Next we consider the discretization of the
minimum distance $\D_{\alpha,p}(g,h; \SO(2))$.
As before, we suppose access to samples $g(t_{j_1},t_{j_2})$
and $h(t_{j_1},t_{j_2})$ of $g$ and $h$
on a Cartesian grid.

By abuse of notation, we denote
by $\what{G}_{\tau}(r)$
the Fourier transform of the Gaussian $G_\tau$
at any $(\xi_1,\xi_2)$ with $r = \sqrt{\xi_1^2 + \xi_2^2}$.
Also abusing notation, denote by $\what{g}(r,\theta)$
and $\what{h}(r,\theta)$ the functions $\what{g}$
and $\what{h}$ in polar coordinates.

For $0 < \tau \le \tau_0$,
define the inner product function $Q_\tau$ by
\begin{align}
Q_\tau(\psi; g,h) = \int_{0}^{2 \pi} \int_{0}^{\infty}  \what{g}(r,\theta)
            \what{h}(r,\theta - \psi)^* \, \what{G}_{\tau}(r)^2
            \, r  dr \, d\theta,
\end{align}
where here ${}^*$ denotes complex conjugation.
Then $Q_\tau(\psi; g, h)$ is the inner product
between $G_\tau \ast g$
and the rotation of $G_\tau \ast h$ by $\psi$.

We introduce several key definitions and observations,
which we will be able to convert into
an efficient algorithm:

\begin{enumerate}

\item
For $r > 0$, write the Fourier expansions
\begin{align}
\what{g}(r,\theta) = \sum_{j} a_j(r) e^{\i j \theta},
\end{align}
and
\begin{align}
\what{h}(r,\theta) = \sum_{j} b_j(r) e^{\i j \theta},
\end{align}
for $0 \le \theta < 2 \pi$.

\item
For each fixed $r$ and $\psi$,
\begin{align}
\int_{0}^{2 \pi} \what{g}(r,\theta)
            \what{h}(r,\theta - \psi)^* \,  d\theta
= \sum_{j} a_j(r) b_j(r)^* e^{\i j \psi}.
\end{align}

\item
The Fourier expansion of $Q_{\tau}$
may then be written as follows:
\begin{align}
Q_\tau(\psi; g,h) &=  \int_{0}^{\infty} \int_{0}^{2 \pi} \what{g}(r,\theta)
            \what{h}(r,\theta - \psi)^* \, \what{G}_{\tau}(r)^2
            \,  d\theta \, r  dr
\nonumber \\
&= \int_{0}^{\infty} \left( \sum_{j} a_j(r) b_j(r)^* e^{\i j \psi}\right)
        \what{G}_{\tau}(r)^2 \, r dr
\nonumber \\
&= \sum_{j} \left(\int_{0}^{\infty} a_j(r) b_j(r)^*
        \what{G}_{\tau}(r)^2 \, r  dr \right) e^{\i j \psi}.
\end{align}
That is, $Q_\tau$ has Fourier coefficients
\begin{align}
\what{Q}_\tau[j] = \int_{0}^{\infty} a_j(r) b_j(r)^*
        \what{G}_{\tau}(r)^2 \, r  dr.
\end{align}

\item
Let $\D_{\alpha,p}(g,h; \, \psi)$
denote the distance between $g$ and the rotation of $h$
by $\psi$.
Then
\begin{align}
\D_{\alpha,p}(g,h; \, \psi)^p
= \sum_{k=0}^{\infty} \tau_k^{p \alpha / 2}
    \left( \|G_{\tau_\ell} \ast g\|_{L^2}^2 + \|G_{\tau_\ell} \ast h\|_{L^2}^2
        - 2 Q_{\tau_\ell}(\psi; g, h) \right)^{p/2}.
\end{align}

\end{enumerate}

We now describe the algorithm.
Let $R = n / (2L)$; this is the maximum frequency
resolvable from the observed samples.
Take integers $N_r$, $N_\theta$, and $N_\psi$,
where $N_\psi \ge N_\theta$.
Let $r_k = R \cdot k  / N_r$, $k=0,\dots,N_r$,
be the radial values; let $\theta_\ell = 2 \pi \ell / N_\theta$,
$\ell = 0,\dots, N_\theta - 1$, be the interpolation angles,
which we assume without loss of generality to be even;
and let $\psi_m = 2 \pi m / N_\psi$, $m=0,\dots,N_\psi-1$,
be the angles at which we will approximate the $Q_{\tau_\ell}(\psi_m)$.
Following the discussion in Section 2.2 of \cite{barnett2017rapid},
we take $N_\theta = O(n)$, and $N_r = O(n)$.
We can take $N_\psi$
substantially larger than $N_\theta$
without incurring additional asymptotic cost
(though the constant obviously grows);
see Remark \ref{rmk:finer_grid} below.

\begin{enumerate}

\item
Using a non-uniform fast Fourier transform, or NUFFT \cite{barnett2021aliasing,
barnett2019parallel, dutt1993fast, dutt1995fast, greengard2004accelerating,
fessler2003nonuniform},
evaluate all sums
\begin{align}
\wtilde{g}[k,\ell]
= \frac{L^2}{n^2}\sum_{j_1 = 1}^{n} \sum_{j_2 = 1}^{n}
    g(t_{j_1},t_{j_2}) e^{- 2 \pi \i (t_{j_1} r_k \cos(\theta_\ell) + t_{j_2} r_k \sin(\theta_\ell))}
\end{align}
at cost $O(n^2 \log(n))$.
These values approximate $\what{g}(r_k, \theta_\ell)$,
the 2D Fourier transform of $g$ at the points $(r_k \cos(\theta_\ell), r_k \sin(\theta_\ell))$.
We do the same for $h$, with values $\wtilde{h}[k,\ell] \approx \what{h}(r_k,\theta_\ell)$.

\begin{rmk}
In our implementation, we use the FINUFFT from the Flatiron Institute
\cite{barnett2021aliasing, barnett2019parallel}.
\end{rmk}

\item
For each $k=0,\dots,N_r$,
evaluate the discrete Fourier coefficients $\wtilde{a}[k,j]$
of $\wtilde{g}[k,\cdot]$, and the discrete Fourier coefficients
$\wtilde{b}[k,j]$ of $\wtilde{h}[k,\cdot]$,
$j = -N_\theta/2,\dots,N_\theta/2$.
Then $\wtilde{a}[k,j] \approx a_j(r_k)$
and $\wtilde{b}[k,j] \approx b_j(r_k)$,
for $k=0,\dots,N_r$ and
$j=-N_\theta/2,\dots,N_\theta/2$.
Using an FFT, for each $k$ the cost of
computing all coefficients $\wtilde{a}[k,j]$
and $\wtilde{b}[k,j]$, $|j| \le N_\theta/2$,
is $O(N_\theta \log(N_\theta)) = O(n \log(n))$,
and so the total cost over all $N_r = O(n)$
radii $r_k$ is $O(n^2 \log(n))$.

\item
For each $\ell$
and $j$
evaluate the approximate integral
\begin{align}
\wtilde{Q}[j,\ell]
= \frac{R}{N_r}\sum_{k=0}^{N_r} \wtilde{a}[k,j] \wtilde{b}[k,j]^* \what{G}_{\tau_\ell}(r_k)^2 r_k
\approx \what{Q}_{\tau_\ell}[j].
\end{align}
For each $j$, the  sum may be evaluated at cost $O(N_r) = O(n)$;
and since there are $N_\theta = O(n)$ Fourier coefficients $j$
and $M_n = O(\log(n))$ levels $\ell$,
the total cost is $O(n^2 \log(n))$.

Having computed approximations to the Fourier
coefficients $\what{Q}_{\tau_\ell}[j]$,
for each $\ell = 0,\dots,M_n$,
we can approximate the inner products $Q_{\tau_\ell}(\psi_m; g, h)$,
$m=0,\dots,N_\psi-1$,
at cost $O(N_\psi \log(N_\psi))$,
by first zero-padding the Fourier coefficients
$\what{Q}_{\tau_\ell}[j]$ and then computing an IFFT.
Repeating this computation for all $M_n$ levels
incurs a total cost of $O(N_\psi \log(N_\psi) \cdot M_n)$.

\begin{rmk}
\label{rmk:finer_grid}
Because the cost of this step is
$O(N_\psi \log(N_\psi) \cdot M_n)$
and $M_n = O(\log(n))$,
we can take $N_\psi$
to be substantially larger than $n$
(nearly on the order of $O(n^2)$)
without increasing the asymptotic cost $O(n^2 \log(n))$ of the algorithm,
allowing for a finer resolution alignment.
\end{rmk}

\item
All $M_n$ norms $\|G_{\tau_\ell} \ast g\|_{L^2}$
and $\|G_{\tau_\ell} \ast h\|_{L^2}$,
$\ell = 0,\dots,M$, can be approximated at cost $O(n^2 \log(n))$:
first compute the 2D FFTs of $g$ and $h$
at cost $O(n^2 \log(n))$,
and then, at each level $\ell = 0,\dots,M_n$,
multiply the DFTs of $g$ and $h$ by the Fourier
coefficients of $G_{\tau_\ell}$
at total cost $O(n^2 \log(n))$;
and sum the coefficients for each level
at cost $O(n^2 \log(n))$.
Alternatively, the norms can be computed
using the samples of
$\what{g}$ and $\what{h}$ already computed on the polar grid.

We can then approximate all $N_{\psi}$ rotated distances
by using the formula
\begin{align}
\D_{\alpha,p}(g,h; \, \psi_m)^p
= \sum_{k=0}^{M_n} \tau_k^{p \alpha / 2}
    \left( \|G_{\tau_\ell} \ast g\|_{L^2}^2 + \|G_{\tau_\ell} \ast h\|_{L^2}^2
        - 2 Q_{\tau_\ell}(\psi_m; g, h) \right)^{p/2},
\end{align}
at an additional cost of $O(N_\psi \log(n))$.
The entire cost of the calculation of
all rotated distances, therefore, is $O(n^2 \log(n))$.

\end{enumerate}

\section{Robustness to noise in 2D}

Suppose the random variables $Z[j_1,j_2]$, $0 \le j_1,j_2 \le n-1$
are independent, zero mean sub-Gaussians with respective
sub-Gaussian norms $\sigma[j_1,j_2] = \|Z[j_1,j_2]\|_{\psi_2}$;
and suppose
\begin{align}
\frac{1}{n} \sum_{j=0}^{n-1} \sigma[j_1,j_2]^2 \le \sigma^2
\end{align}
for all $n$, where $\sigma > 0$.
Define the random variable
\begin{align}
\Pi_n \equiv \|Z\|_{\gamma_p^{\alpha}}
= \left( \sum_{\ell=0}^{M} \tau_\ell^{p \alpha/2} S(\tau_\ell)^p \right)^{1/p}.
\end{align}
We will bound the sub-Gaussian norm of $\Pi_n$:

\begin{thm}
\label{thm:main_subgaussian}
The random variable $\Pi_n = \|Z\|_{\gamma_p^{\alpha}}$ has sub-Gaussian
norm bounded as follows:
\begin{align}
\|\Pi_n\|_{\psi_2} \le C \cdot \frac{\tau_0^{\alpha}}{1 - \omega^{1/2}}
    \cdot L \cdot \sqrt{p} \cdot \sigma
    \cdot \Delta_{\alpha,p}(L / n \sqrt{\tau_0}),
\end{align}
where $C>0$ is a universal constant.
\end{thm}

\begin{rmk}
In particular, when $\alpha > 1$,
the sub-Gaussian norm decays like $1/n$,
which is the inverse square root decay
($n$ being the square root of the total number of
observations) we would expect from
averaging independent random variables.
However, when $0 < \alpha \le 1$, the rate of decay
is worse:
at $\alpha = 1$, the decay is like $(1 + \log(n))^{1/p} / n$,
whereas when $0 < \alpha < 1$, it decays
at the even slower rate of $1 / n^\alpha$.
This is because with smaller values,
the weights $\tau_k^{p \alpha / 2}$ decay slower,
and hence place relatively more weight on the
higher-level Gaussian filters,
which do not filter the noise as aggressively.

\end{rmk}

\begin{rmk}
By Proposition 2.6.6 of \cite{vershynin2025high},
we have the concentration inequality
\begin{align}
P(\|Z\|_{\gamma_{p}^{\alpha}} > t)
\le 2 e^{- t^2 / \|\Pi_n\|_{\psi_2}^2}.
\end{align}
It follows immediately
(using Borel-Cantelli) that $\|Z\|_{\gamma_{p}^{\alpha}}$ converges to $0$
almost surely as $n \to \infty$.
Furthermore, we also have the bound on the expected value of $\|Z\|_{\gamma_{p}^{\alpha}}$:
\begin{align}
\EE\left[ \|Z\|_{\gamma_{p}^{\alpha}} \right] \le
C \cdot \frac{\tau_0^{\alpha}}{1 - \omega^{1/2}}
    \cdot L \cdot \sqrt{p} \cdot \sigma
    \cdot \Delta_{\alpha,p}(L / n \sqrt{\tau_0}),
\end{align}
which also follows immediately from
Proposition 2.6.6 of \cite{vershynin2025high},
with a possibly different universal constant $C$.
\end{rmk}

We now prove Theorem \ref{thm:main_subgaussian}.

\begin{proof}[Proof of Theorem \ref{thm:main_subgaussian}]
We have
\begin{align}
\what{Z}[k_1,k_2]
    = \frac{L^2}{n^2} \sum_{j_1=0}^{n-1} \sum_{j_2=0}^{n-1}Z[j_1,j_2]
        e^{-2 \pi \i (k_1 t_{j_1}  + k_2 t_{j_2} )/L},
\end{align}
The $Z[j_1,j_2]$ are independent sub-Gaussians with mean zero,
and by Proposition 2.7.1 in \cite{vershynin2025high} there is a universal $C>0$ such that
\begin{align}
\left\| \what{Z}[k_1,k_2]  \right\|_{\psi_2}^2
\le C \frac{L^4}{n^4} \sum_{j_1=0}^{n-1} \sum_{j_2=0}^{n-1} \| Z[j_1,j_2] \|_{\psi_2}^2
= C \frac{L^4}{n^2} \left(\frac{1}{n^2} \sum_{j_1=0}^{n-1} \sum_{j_2=0}^{n-1} \sigma_j^2 \right)
\le C L^4 \frac{\sigma^2}{n^2}.
\end{align}
(Note that while Proposition 2.7.1 in \cite{vershynin2025high}
is stated for real-valued random variables, it is not hard
to extend to complex-valued random variables.)

Therefore, the sub-Gaussian norm of
\begin{align}
\what{Z}_\tau[k_1,k_2] = \what{Z}[k_1,k_2] \what{G}_\tau(k_1/L,k_2/L)
\end{align}
is similarly bounded above:
\begin{align}
\left\| \what{Z}_\tau[k_1,k_2]  \right\|_{\psi_2}^2
= C L^4 \frac{\sigma^2}{n^2} \what{G}_\tau(k_1/L,k_2/L)^2.
\end{align}
For $\tau \le \tau_0$,
we now bound the sub-Gaussian norm of the random variable
\begin{align}
S(\tau)^2
= \frac{1}{L^2}\sum_{k_1=-n/2}^{n/2-1} \sum_{k_2=-n/2}^{n/2-1} | \what{Z}_{\tau}[k_1,k_2] |^2.
\end{align}

Since $\|S(\tau)^2\|_{\psi_1} = \|S(\tau)\|_{\psi_2}^2$,
and since $\|\cdot \|_{\psi_1}$ is a norm,
therefore,
\begin{align}
\label{eq:S_tau_bound-1}
\|S(\tau)\|_{\psi_2}^2
&= \|S(\tau)^2\|_{\psi_1}
\nonumber \\
&\le \frac{1}{L^2}\sum_{k_1=-n/2}^{n/2-1}
    \sum_{k_2=-n/2}^{n/2-1} \left\| | \what{Z}_{\tau}[k_1,k_2] |^2 \right\|_{\psi_1}
\nonumber \\
&= \frac{1}{L^2}\sum_{k_1=-n/2}^{n/2-1}
    \sum_{k_2=-n/2}^{n/2-1} \left\| | \what{Z}_{\tau}[k_1,k_2] | \right\|_{\psi_2}^2
\nonumber \\
&\le \frac{1}{L^2}\sum_{k_1=-n/2}^{n/2-1}
    \sum_{k_2=-n/2}^{n/2-1} C L^4 \frac{\sigma^2}{n^2} \what{G}_\tau(k_1/L,k_2/L)^2
\nonumber \\
&= C \sigma^2 \frac{L^2}{n^2}\sum_{k_1=-n/2}^{n/2-1}
    \sum_{k_2=-n/2}^{n/2-1} \what{G}_\tau(k_1/L,k_2/L)^2.
\end{align}

To bound this sum, we factor
\begin{align}
\what{G}_\tau(k_1/L,k_2/L)^2
= e^{- 2 \pi^2 \tau (k_1/L)^2} e^{- 2\pi^2 \tau (k_2/L)^2},
\end{align}
which allows us to write the bound in \eqref{eq:S_tau_bound-1}
as
\begin{align}
\|S(\tau)\|_{\psi_2}^2
\le C L^4\frac{\sigma^2}{n^2}
    \left(\frac{1}{L}\sum_{k=-n/2}^{n/2-1} e^{- 2 \pi^2 \tau (k/L)^2}\right)^2.
\end{align}

To bound this: we write
\begin{align}
\frac{1}{L}\sum_{k=-n/2}^{n/2-1} e^{-2 \pi^2 \tau (k/L)^2}
&\le \frac{1}{L} + \frac{2}{L} \sum_{k=1}^{n/2} e^{-2 \pi^2 \tau (k/L)^2}
\nonumber \\
&\le \frac{1}{L} + \frac{2}{L} \sum_{k=1}^{\infty} e^{-2 \pi^2 \tau (k/L)^2}
\nonumber \\
&\le \frac{1}{L}
    + 2 \sum_{k=1}^{\infty} \int_{(k-1)/L}^{k/L} e^{-2 \pi^2 \tau \xi^2} \, d\xi
\nonumber \\
&= \frac{1}{L}
    + 2 \int_{0}^{\infty} e^{-2 \pi^2 \tau \xi^2} \, d\xi
\nonumber \\
&= \frac{1}{L}
    + \frac{2}{\sqrt{\tau}} \int_{0}^{\infty} e^{-2 \pi^2 \nu^2} \, d\nu.
\end{align}

Therefore,
\begin{align}
\label{eq:S_tau_bound_1}
\|S(\tau)\|_{\psi_2}^2
\le C L^4\frac{\sigma^2}{n^2}
    \left(\frac{1}{L}
    + \frac{2}{\sqrt{\tau}} \int_{0}^{\infty} e^{-2 \pi^2 \nu^2} \, d\nu\right)^2
\le C L^4\frac{\sigma^2}{n^2} \frac{1}{\tau},
\end{align}
for a possibly different constant $C>0$,
and where we have used that
$\sqrt{\tau} \le \sqrt{\tau_0} \le L$
(indeed, $L = L_0 + 2\sqrt{\tau_0 \log(1/\epsilon)}$
for small $\epsilon > 0$).

Of course, we also have the more trivial bound
\begin{align}
\frac{1}{L}\sum_{k=-n/2}^{n/2-1} e^{-2 \pi^2 \tau (k/L)^2}
\le \frac{n}{L},
\end{align}
from which it follows that
\begin{align}
\label{eq:S_tau_bound_2}
\|S(\tau)\|_{\psi_2}^2
\le C L^4\frac{\sigma^2}{n^2}
    \left(\frac{1}{L}\sum_{k=-n/2}^{n/2-1} e^{- 2 \pi^2 \tau (k/L)^2}\right)^2
\le C L^4\frac{\sigma^2}{n^2} \frac{n^2}{L^2}
= C L^2 \sigma^2.
\end{align}

To bound the sub-Gausian norm of $\Pi_n$
itself,
we now have the following lemma,
whose proof we provide for
the convenience of the reader.

\begin{lem}
\label{lem:subgaussian_pnorm}
Suppose $X_1,\dots,X_m$ are
sub-Gaussian random variables,
$p \ge 1$,
and $X = \left(\sum_{i=1}^{m} |X_i|^p \right)^{1/p}$.
Then $X$ is sub-Gaussian, with
\begin{align}
\|X\|_{\psi_2} \le C \sqrt{p} \left(\sum_{i=1}^{m} \|X_i\|_{\psi_2}^p \right)^{1/p},
\end{align}
where $C>0$ is a universal constant.
\end{lem}

\begin{proof}[Proof of Lemma \ref{lem:subgaussian_pnorm}]
From Proposition 2.6.1 in \cite{vershynin2025high},
the sub-Gaussian norm $\|X\|_{\psi_2}$
is equivalent to the norm
$\|X\|_{\psi_2}^{(1)} = \sup_{r \ge 1} r^{-1/2} (\EE[|X|^r])^{1/r}$,
in the sense that $\|X\|_{\psi_2}^{(1)} / \|X\|_{\psi_2}$
is bounded above and below by universal constants.
Fix $r \ge 1$. We consider two cases.

\item
\paragraph{Case 1: $r \le p$.}
First, as just remarked, by Proposition 2.6.1 in \cite{vershynin2025high}
there is a universal $C>0$
such that
\begin{align}
\EE  \left[ |X_i|^p \right]
\le C^p p^{p/2} \|X_i\|_{\psi_2}^p.
\end{align}
Because $r/p \le 1$, by Jensen's inequality we have
\begin{align}
\EE\left[ X^r\right]
&\le \EE\left[ \left(\sum_{i=1}^{m} |X_i|^p \right)^{r/p}\right]
\nonumber \\
&\le \left[ \EE \left(\sum_{i=1}^{m} |X_i|^p \right)\right]^{r/p}
\nonumber \\
&= \left[ \sum_{i=1}^{m} \EE  \left[ |X_i|^p \right] \right]^{r/p}
\nonumber \\
&\le \left[ \sum_{i=1}^{m} C^{p} p^{p/2} \|X_i\|_{\psi_2}^p \right]^{r/p}
\nonumber \\
&= C^{r} p^{r/2} \left[ \sum_{i=1}^{m} \|X_i\|_{\psi_2}^p \right]^{r/p},
\end{align}
and taking $r$-th roots, dividing by $\sqrt{r}$,
and using that $r \ge 1$ and hence $r^{-1/2} \le 1$, we have
\begin{align}
r^{-1/2} (\EE[|X|^r])^{1/r}
\le r^{-1/2} C \sqrt{p} \left[ \sum_{i=1}^{m} \|X_i\|_{\psi_2}^p \right]^{1/p}
\le C \sqrt{p} \left[ \sum_{i=1}^{m} \|X_i\|_{\psi_2}^p \right]^{1/p},
\end{align}
which is the desired bound.

\item
\paragraph{Case 2: $r > p$.}
For clarity, let $Y_i = |X_i|^p$,
and let $q = r/p > 1$.
Again, using
Proposition 2.6.1 in \cite{vershynin2025high},
$(\EE[|X_i|^{r}])^{1/r} \le C \sqrt{r} \|X_i\|_{\psi_2}$
for universal $C>0$.

Using, for instance, Minkowski's inequality
for integrals (e.g.\ Theorem 6.19
of \cite{folland1999real}),
we have the following inequality:
\begin{align}
(\EE[X^r])^{1/r}
&= \left( \EE\left[ \left(\sum_{i=1}^{n} |X_i|^p \right)^{r/p} \right] \right)^{1/r}
\nonumber \\
&= \left(\left( \EE\left[ \left(\sum_{i=1}^{n} Y_i \right)^{q} \right] \right)^{1/q} \right)^{q/r}
\nonumber \\
&\le \left(\sum_{i=1}^{n} (\EE[Y_i^q])^{1/q} \right)^{q/r}
\nonumber \\
&= \left(\sum_{i=1}^{n} (\EE[|X_i|^{r}])^{p/r} \right)^{1/p}
\nonumber \\
&\le \left(\sum_{i=1}^{n} C^p r^{p/2} \|X_i\|_{\psi_2}^p \right)^{1/p}
\nonumber \\
&= C r^{1/2} \left(\sum_{i=1}^{n} \|X_i\|_{\psi_2}^p \right)^{1/p},
\end{align}
and so
\begin{align}
r^{-1/2} (\EE[X^r])^{1/r}
\le C \left(\sum_{i=1}^{n} \|X_i\|_{\psi_2}^p \right)^{1/p},
\end{align}
which is the desired bound.

\end{proof}

Using Lemma \ref{lem:subgaussian_pnorm}
we immediately get the bound
\begin{align}
\| \Pi_n \|_{\psi_2}
= \left\| \left(\sum_{\ell=0}^{M} \tau_\ell^{p \alpha/2} S(\tau_\ell)^p \right)^{1/p}
        \right\|_{\psi_{2}}
\le C \sqrt{p}
    \left(\sum_{\ell=0}^{M} \tau_\ell^{p \alpha/2} \| S(\tau_\ell) \|_{\psi_{2}}^p \right)^{1/p}
\le C \sqrt{p}
    \left(\sum_{\ell=0}^{\infty} \tau_\ell^{p \alpha/2} \| S(\tau_\ell) \|_{\psi_{2}}^p \right)^{1/p},
\end{align}
for universal $C>0$.

Let
\begin{align}
N = \lfloor 2\log_{1/\omega}(n \sqrt{\tau_0} / L) \rfloor.
\end{align}
We break the proof into the three regimes of $\alpha$
appearing in the definition of $\Delta_{\alpha,p}$.

\item
\paragraph{Case 1: $0 < \alpha < 1$.}
First, break the series into two parts:
\begin{align}
\sum_{\ell=0}^{\infty} \tau_\ell^{p \alpha/2} \| S(\tau_\ell) \|_{\psi_{2}}^p
&=
\sum_{\ell=0}^{N} \tau_\ell^{p \alpha/2} \| S(\tau_\ell) \|_{\psi_{2}}^p
    + \sum_{\ell=N+1}^{\infty} \tau_\ell^{p \alpha/2} \| S(\tau_\ell) \|_{\psi_{2}}^p.
\end{align}

For the first sum,
we use \eqref{eq:S_tau_bound_1},
which implies
\begin{align}
\|S(\tau_\ell)\|_{\psi_2}^p
\le C^{p/2} L^{2p} \frac{\sigma^{p}}{n^{p}} \frac{1}{\tau_\ell^{p/2}},
\end{align}
and hence
\begin{align}
\label{eq:504002-1}
\sum_{\ell=0}^{N} \tau_\ell^{p \alpha/2} \| S(\tau_\ell) \|_{\psi_{2}}^p
&\le C^{p/2} L^{2p} \frac{\sigma^{p}}{n^{p}}
        \sum_{\ell=0}^{N} \tau_\ell^{p \alpha/2} \frac{1}{\tau_\ell^{p/2}}
\nonumber \\
&\le C^{p/2} L^{2p} \frac{\sigma^{p}}{n^{p}}
        \sum_{\ell=0}^{N} \tau_\ell^{p (\alpha-1)/2}
\nonumber \\
&= C^{p/2} L^{2p} \frac{\sigma^{p}}{n^{p}} \tau_0^{p (\alpha-1)/2}
        \sum_{\ell=0}^{N} \omega^{\ell p (\alpha-1)/2}
\nonumber \\
&= C^{p/2} L^{2p} \frac{\sigma^{p}}{n^{p}} \tau_0^{p (\alpha-1)/2}
        \frac{\omega^{N p (\alpha-1)/2}}{1 - \omega^{p (1-\alpha)/2}}
\nonumber \\
&\le C^{p/2} L^{2p} \frac{\sigma^{p}}{n^{p}} \tau_0^{p (\alpha-1)/2}
        \frac{n^{p (1-\alpha)} L^{p (\alpha-1)} \tau_0^{p (1-\alpha)/2}}{1 - \omega^{p (1 - \alpha)/2}}
\nonumber \\
&\le C^{p/2} L^{p(\alpha+1)} \frac{\sigma^{p}}{n^{p\alpha}}
        \frac{1}{1 - \omega^{p (1 - \alpha)/2}}.
\nonumber \\
&\le C^{p/2} L^{p(\alpha+1)} \tau_0^{-p \alpha/2} \frac{\sigma^{p}}{n^{p\alpha}}
        \frac{\tau_0^{p\alpha/2}}{1 - \omega^{p/2}} \cdot \frac{1}{1-\alpha}.
\end{align}

Using the bound
\eqref{eq:S_tau_bound_2},
or equivalently,
$\|S(\tau)\|_{\psi_2}^p \le C^{p/2} L^{p} \sigma^{p}$,
the second sum can be bounded by
\begin{align}
\label{eq:504002-2}
\sum_{\ell=N+1}^{\infty} \tau_\ell^{p \alpha/2} \| S(\tau_\ell) \|_{\psi_{2}}^p
&\le C^{p/2} L^{p} \sigma^{p}\sum_{\ell=N+1}^{\infty} \tau_\ell^{p \alpha/2}
\nonumber \\
&= C^{p/2} L^{p} \sigma^{p} \tau_0^{p \alpha / 2}\sum_{\ell=N+1}^{\infty} \omega^{\ell p \alpha/2}
\nonumber \\
&\le C^{p/2} L^{p} \sigma^{p} \tau_0^{p \alpha / 2}
    \frac{\omega^{(N+1) p \alpha / 2}}{1- \omega^{p \alpha / 2}}
\nonumber \\
&\le C^{p/2} L^{p} \sigma^{p} \tau_0^{p \alpha / 2}
    \frac{n^{-p \alpha} L^{p \alpha} \tau_0^{-p \alpha/2}}{1 - \omega^{p \alpha/2}}
\nonumber \\
&= C^{p/2} L^{p(\alpha+1)} \tau_0^{-p \alpha/2} \frac{\sigma^{p} }{n^{p \alpha}}
    \frac{\tau_0^{p \alpha/2}}{1- \omega^{p \alpha / 2}}
\nonumber \\
&\le C^{p/2} L^{p(\alpha+1)} \tau_0^{-p \alpha/2} \frac{\sigma^{p} }{n^{p \alpha}}
    \frac{\tau_0^{p \alpha/2}}{1- \omega^{p / 2}} \cdot \frac{1}{\alpha}.
\end{align}

Adding \eqref{eq:504002-1} and \eqref{eq:504002-2}
\begin{align}
p^{-p/2} \| \Pi_n \|_{\psi_2}^p
&\le C^{p/2} L^{p(\alpha+1)} \tau_0^{-p \alpha/2} \frac{\sigma^{p}}{n^{p\alpha}}
        \frac{\tau_0^{p\alpha/2}}{1 - \omega^{p/2}} \cdot \frac{1}{1-\alpha}
    + C^{p/2} L^{p(\alpha+1)} \tau_0^{-p \alpha/2} \frac{\sigma^{p} }{n^{p \alpha}}
    \frac{\tau_0^{p \alpha/2}}{1- \omega^{p / 2}} \cdot \frac{1}{\alpha}
\nonumber \\
&= C^{p/2} L^{p(\alpha+1)} \tau_0^{-p \alpha/2} \frac{\sigma^{p}}{n^{p\alpha}}
    \frac{\tau_0^{p\alpha/2}}{1 - \omega^{p/2}} \left( \frac{1}{1-\alpha} + \frac{1}{\alpha} \right)
\nonumber \\
&= C^{p/2} L^{p(\alpha+1)} \tau_0^{-p \alpha/2} \frac{\sigma^{p}}{n^{p\alpha}}
    \frac{\tau_0^{p\alpha/2}}{1 - \omega^{p/2}} \cdot \frac{1}{\alpha-\alpha^2}.
\end{align}
Taking $p$-th roots, and for a different universal $C>0$,
\begin{align}
p^{-1/2} \| \Pi_n \|_{\psi_2}
&\le C \cdot L \cdot \tau_0^{\alpha/2} \left(\frac{1}{1 - \omega^{p/2}} \right)^{1/p}
    \cdot L^\alpha \cdot \tau_0^{- \alpha/2} \cdot \frac{\sigma}{n^{\alpha}}
    \cdot \frac{1}{(\alpha-\alpha^2)^{1/p}}
\nonumber \\
& \le C \cdot L \cdot \sigma \cdot \frac{\tau_0^{\alpha/2}}{1 - \omega^{1/2}}
    \cdot \Delta_{\alpha,p}(L / n \sqrt{\tau_0}),
\end{align}
which is the desired bound.

\item
\paragraph{Case 2: $\alpha = 1$.}

In this case, we again split the sum into two parts:
\begin{align}
\sum_{\ell=0}^{\infty} \tau_\ell^{p /2} \| S(\tau_\ell) \|_{\psi_{2}}^p
&=
\sum_{\ell=0}^{N} \tau_\ell^{p/2} \| S(\tau_\ell) \|_{\psi_{2}}^p
    + \sum_{\ell=N+1}^{\infty} \tau_\ell^{p /2} \| S(\tau_\ell) \|_{\psi_{2}}^p.
\end{align}

Again using \eqref{eq:S_tau_bound_1}, or equivalently
\begin{align}
\|S(\tau_\ell)\|_{\psi_2}^p
\le C^{p/2} L^{2p} \frac{\sigma^{p}}{n^{p}} \frac{1}{\tau_\ell^{p/2}},
\end{align}
and using the inequality \eqref{eq:log_dumb}
which states that $\log(1/t) \ge 1 - t$ for
$0 < t \le 1$,
the first sum can be bounded by
\begin{align}
\label{eq:050402-1}
\sum_{\ell=0}^{N} \tau_\ell^{p /2} \| S(\tau_\ell) \|_{\psi_{2}}^p
&\le C^{p/2} L^{2p} \frac{\sigma^{p}}{n^{p}}
    \sum_{\ell=0}^{N} \tau_\ell^{p /2} \frac{1}{\tau_\ell^{p/2}}
\nonumber \\
&\le C^{p/2} L^{2p} \frac{\sigma^{p}}{n^{p}} N
\nonumber \\
&\le C^{p/2} L^{2p} \frac{\sigma^{p}}{n^{p}} \log_{1/\omega}(n \sqrt{\tau_0} / L)
\nonumber \\
&= C^{p/2} L^{2p} \frac{\sigma^{p}}{n^{p}}
    \frac{\log(n \sqrt{\tau_0} / L)}{\log(1/\omega)}
\nonumber \\
&= C^{p/2} L^{2p} \frac{\sigma^{p}}{n^{p}}
    \frac{(p/2)\log(n \sqrt{\tau_0} / L)}{\log(1/\omega^{p/2})}
\nonumber \\
&= C^{p/2} L^{2p} \frac{\sigma^{p}}{n^{p}}
    \frac{(p/2)\log(n \sqrt{\tau_0} / L)}{1 - \omega^{p/2}}
\end{align}

And again using the bound
$\|S(\tau)\|_{\psi_2}^2
\le C L^2 \sigma^2$
from \eqref{eq:S_tau_bound_2},
the second sum can be bounded by
\begin{align}
\label{eq:050402-2}
\sum_{\ell=N+1}^{\infty} \tau_\ell^{p /2} \| S(\tau_\ell) \|_{\psi_{2}}^p
&\le C^{p/2} L^{p} \sigma^{p} \tau_0^{p /2} \sum_{\ell=N+1}^{\infty} \omega^{\ell p /2}
\nonumber \\
&= C^{p/2} L^{p} \sigma^{p} \tau_0^{p /2} \frac{\omega^{(N+1) p /2}}{1 - \omega^{p/2}}
\nonumber \\
&\le C^{p/2} L^{p} \sigma^{p} \tau_0^{p /2} \frac{\tau_0^{-p/2} n^{-p} L^{p} }{1 - \omega^{p/2}}
\nonumber \\
&= C^{p/2} L^{2p} \frac{\sigma^{p}}{n^p} \frac{ 1}{1 - \omega^{p/2}}
\end{align}

Therefore, adding \eqref{eq:050402-1} and \eqref{eq:050402-2},
\begin{align}
p^{-p/2} \|\Pi_n\|_{\psi_2}^p
&\le C^{p/2} L^{2p} \frac{\sigma^{p}}{n^{p}} \frac{(p/2)\log(n \sqrt{\tau_0} / L)}{1 - \omega^{p/2}}
    + C^{p/2} L^{2p} \frac{\sigma^{p}}{n^p} \frac{ 1}{1 - \omega^{p/2}}
\nonumber \\
&= C^{p/2} \cdot L^{2p} \cdot \tau_0^{- p \alpha/2} \cdot \frac{\sigma^{p}}{n^{p}}
    \cdot \frac{\tau_0^{p \alpha/2}}{1 - \omega^{p/2}} \cdot
        \left( (p/2)\log(n \sqrt{\tau_0} / L) + 1\right)
\end{align}
and taking $p$-th roots, using that $p^{1/p}$ is bounded,
and with a different universal $C>0$, gives
\begin{align}
p^{-1/2}\|\Pi_n\|_{\psi_2}
&\le C \cdot L^{2} \cdot \tau_0^{-  \alpha/2} \cdot \frac{\sigma}{n}
    \cdot \tau_0^{\alpha/2} \cdot \left(\frac{1}{1 - \omega^{p/2}}\right)^{1/p} \cdot
        \left(\log(n \sqrt{\tau_0} / L) + 1\right)^{1/p}
\nonumber \\
&\le C \cdot L^{2} \cdot \tau_0^{-  \alpha/2} \cdot \frac{\sigma}{n}
    \cdot \tau_0^{\alpha/2} \cdot \frac{1}{1 - \omega^{1/2}} \cdot
        \left(\log(n \sqrt{\tau_0} / L) + 1\right)^{1/p}
\nonumber \\
&= C \cdot L \cdot \sigma \cdot  \frac{\tau_0^{\alpha/2} }{1 - \omega^{1/2}} \cdot
    \Delta_{1,p}(L / n\sqrt{\tau_0}),
\end{align}
as desired.

\item
\paragraph{Case 3: $\alpha > 1$.}
We use \eqref{eq:S_tau_bound_1},
or equivalently
\begin{align}
\|S(\tau_\ell)\|_{\psi_2}^p
\le C^{p/2} L^{2p} \frac{\sigma^{p}}{n^{p}} \frac{1}{\tau_\ell^{p/2}},
\end{align}
and write
\begin{align}
\sum_{\ell=0}^{\infty} \tau_\ell^{p \alpha/2} \| S(\tau_\ell) \|_{\psi_{2}}^p
&\le C^{p/2} L^{2p} \frac{\sigma^{p}}{n^{p}}  \sum_{\ell=0}^{\infty} \tau_\ell^{p \alpha/2}
        \frac{1}{\tau_\ell^{p/2}}
\nonumber \\
&= C^{p/2} L^{2p} \frac{\sigma^{p}}{n^{p}}  \sum_{\ell=0}^{\infty} \tau_\ell^{p (\alpha-1)/2}
\nonumber \\
&= C^{p/2} L^{2p} \frac{\sigma^{p}}{n^{p}} \tau_0^{p(\alpha-1)/2}
    \sum_{\ell=0}^{\infty} \omega^{\ell p (\alpha-1)/2}
\nonumber \\
&= C^{p/2} L^{2p} \frac{\sigma^{p}}{n^{p}} \frac{\tau_0^{p(\alpha-1)/2}}{1 - \omega^{p(\alpha-1)/2}}.
\end{align}

Using \eqref{eq:54930201}
from Case 3 of the proof of Theorem \ref{thm:main_projections},
we have
\begin{align}
\left(\frac{1}{1 - \omega^{(\alpha-1)p/2}}\right)^{1/p}
&\le \frac{1}{1 - \omega^{1/2}} \cdot \frac{1}{\min\{1,(\alpha-1)^{1/p}\}}
\end{align}
and so it follows that, for a possibly different universal $C>0$,
\begin{align}
p^{-1/2} \|\Pi_n\|_{\psi_2}
&\le C \cdot \frac{\sigma}{n}
    \cdot L^2  \frac{\tau_0^{ (\alpha - 1)/2} }{(1 - \omega^{(\alpha - 1)p/2})^{1/p}}
\nonumber \\
&\le C \cdot \frac{\tau_0^{\alpha/2}}{1 - \omega^{1/2}} 
    \cdot\tau_0^{-1/2} \cdot \frac{\sigma}{n}
    \cdot L^2   \cdot \frac{1}{\min\{1,(\alpha-1)^{1/p}\}}
\nonumber \\
&= C \cdot L \cdot\sigma\cdot \frac{\tau_0^{\alpha/2}}{1 - \omega^{1/2}} 
    \cdot \Delta_{\alpha,p}(L / n\sqrt{\tau_0}),
\end{align}
which completes the proof of Theorem \ref{thm:main_subgaussian}.

\end{proof}

\begin{figure}[h]
\centering
\includegraphics[scale=.4]{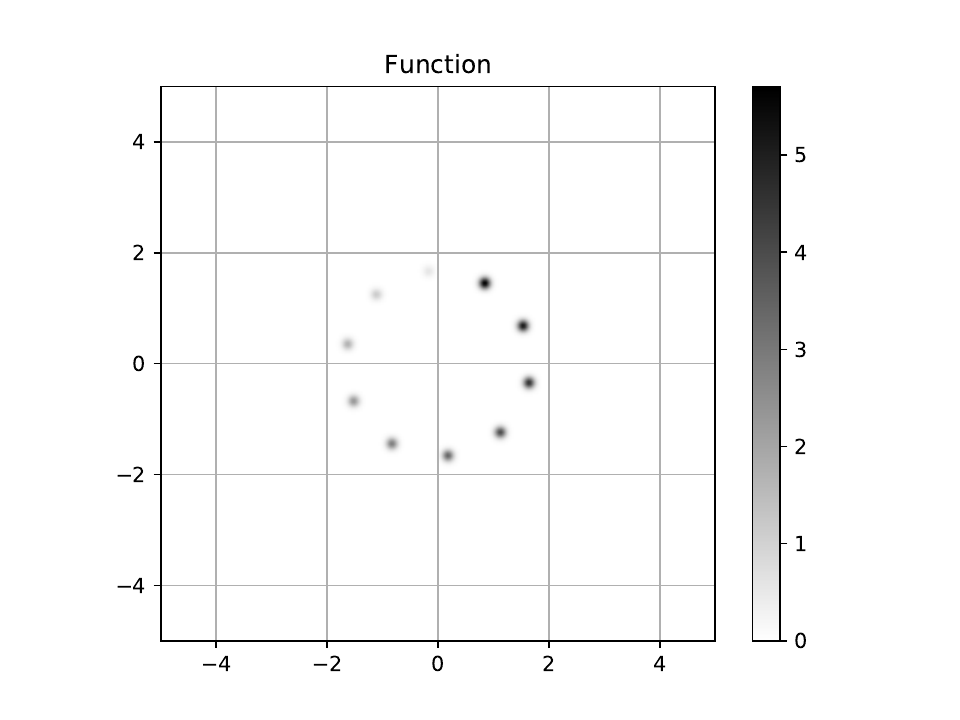}
\includegraphics[scale=.4]{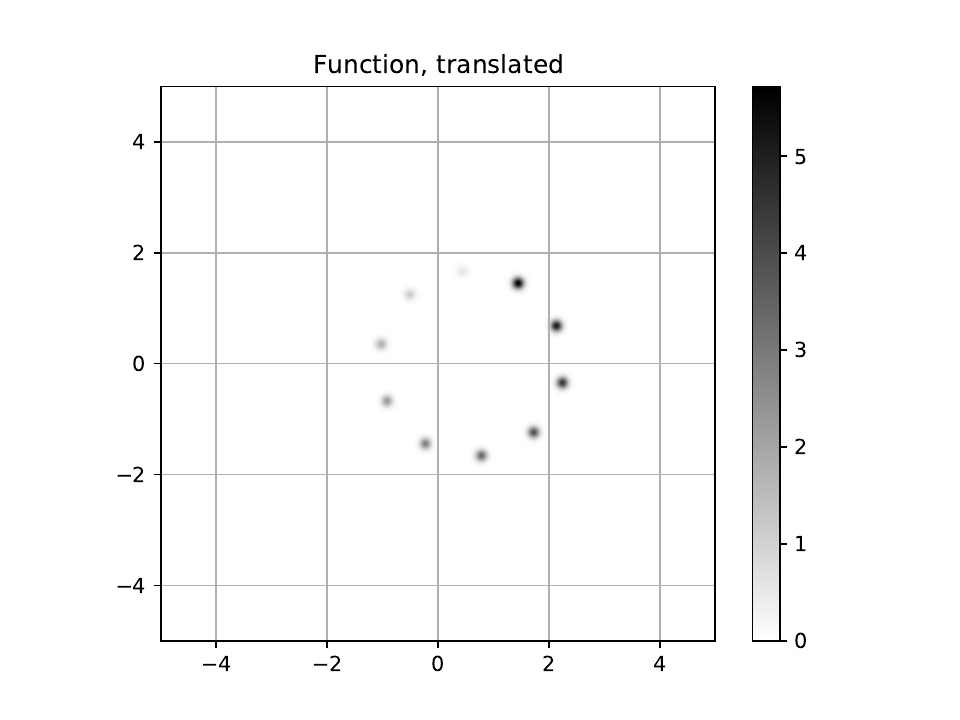}
\caption{The function $f$ described
in Section \ref{sec:transl_rots},
and its translation by $0.6$ to the right.
}
\label{fig:function_ring}
\end{figure}

\begin{figure}[h]
\centering
\includegraphics[scale=.3]{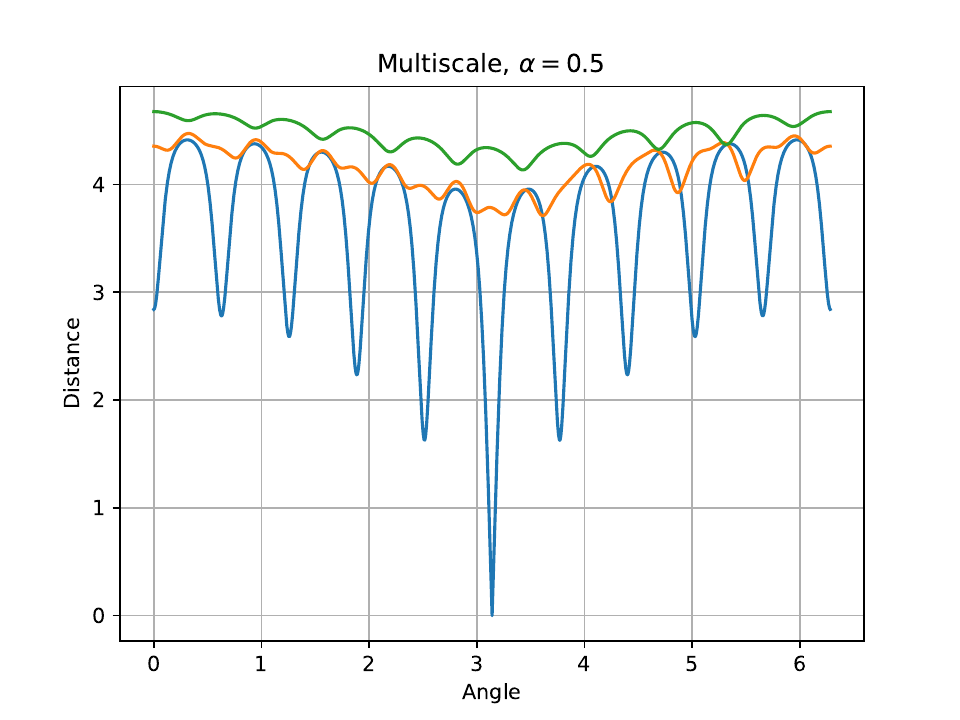}
\includegraphics[scale=.3]{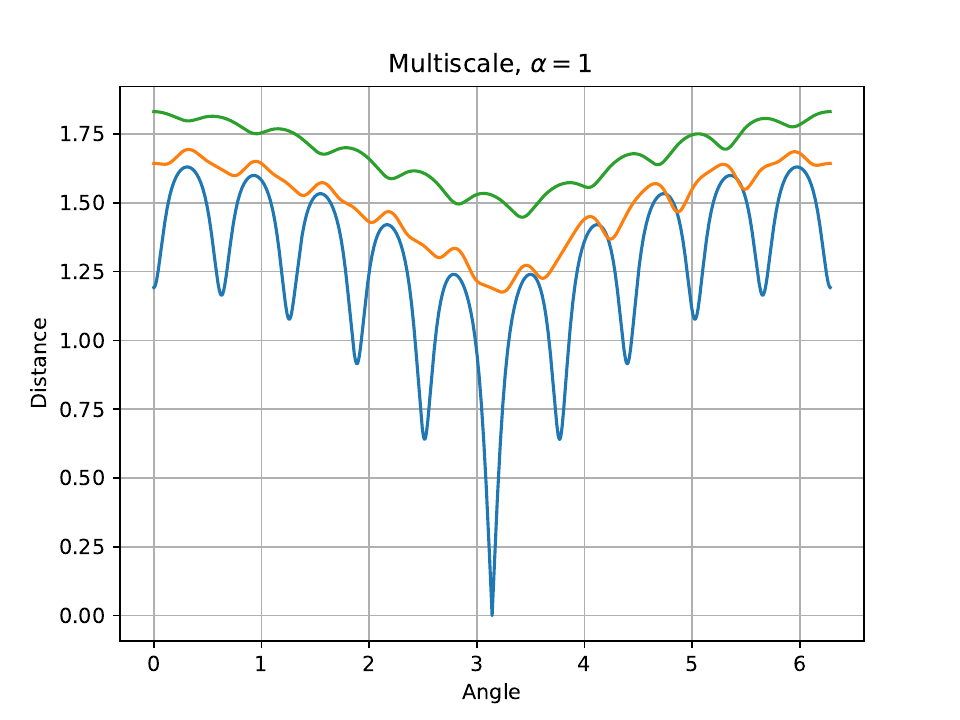}
\includegraphics[scale=.3]{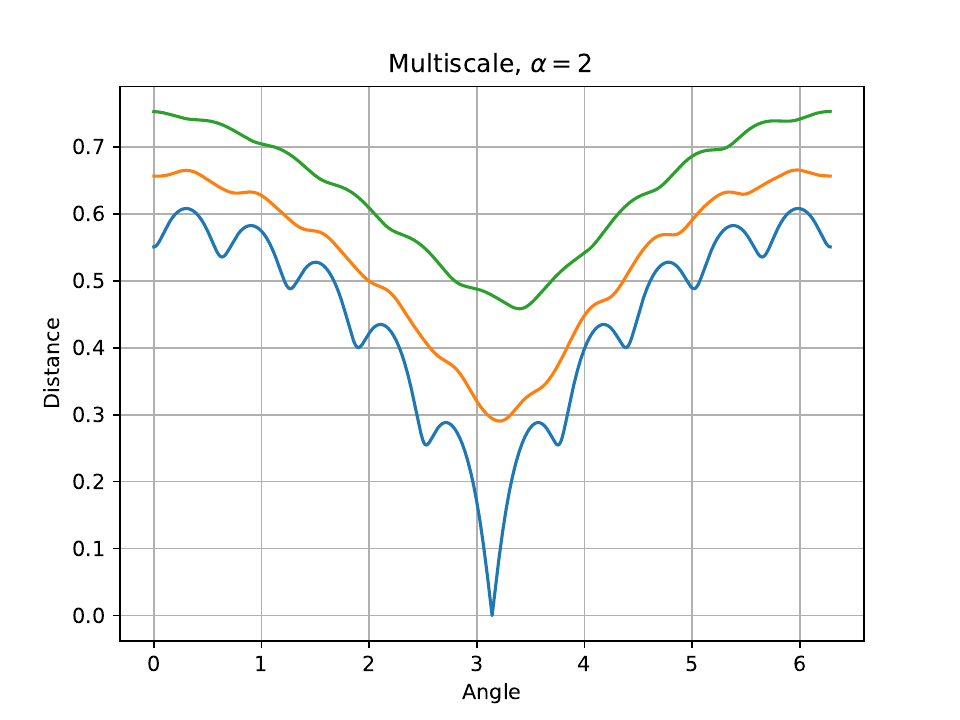}
\\
\includegraphics[scale=.3]{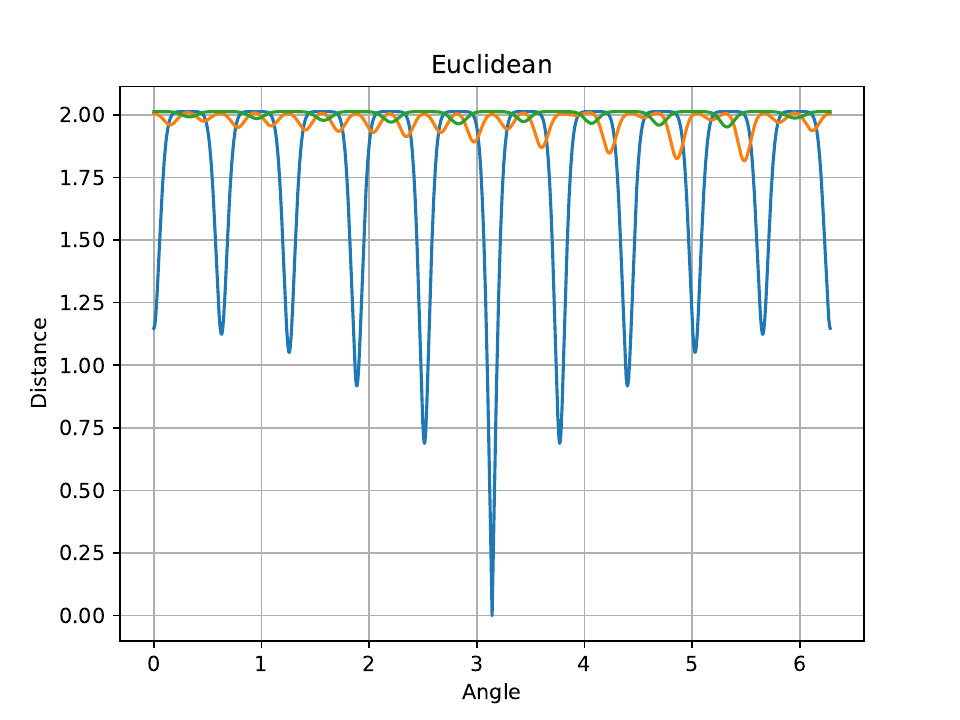}
\includegraphics[scale=.3]{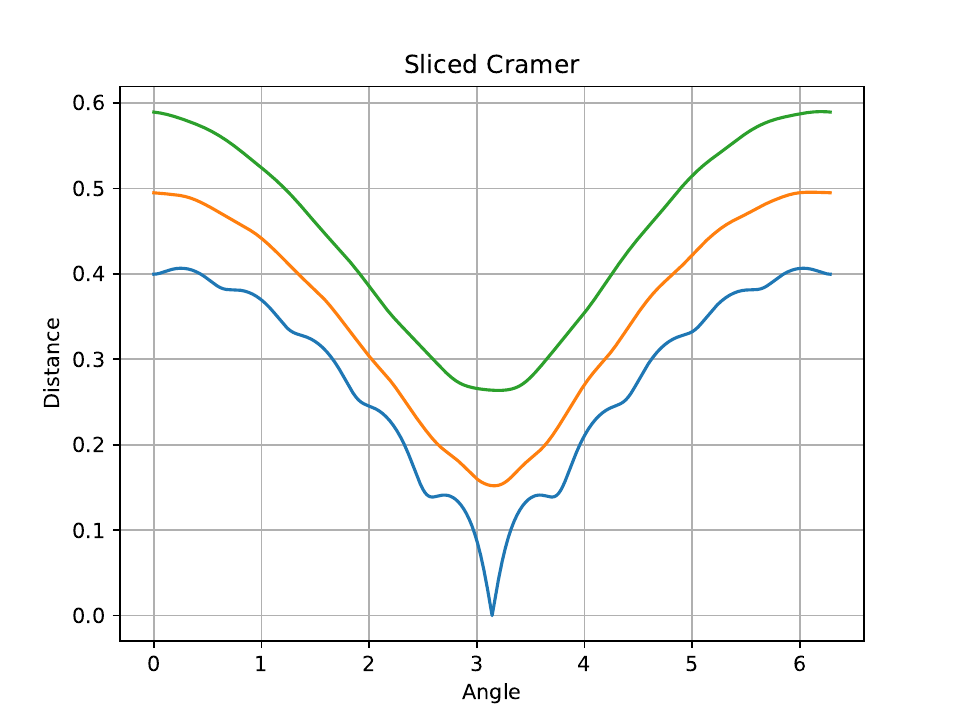}
\includegraphics[scale=.3]{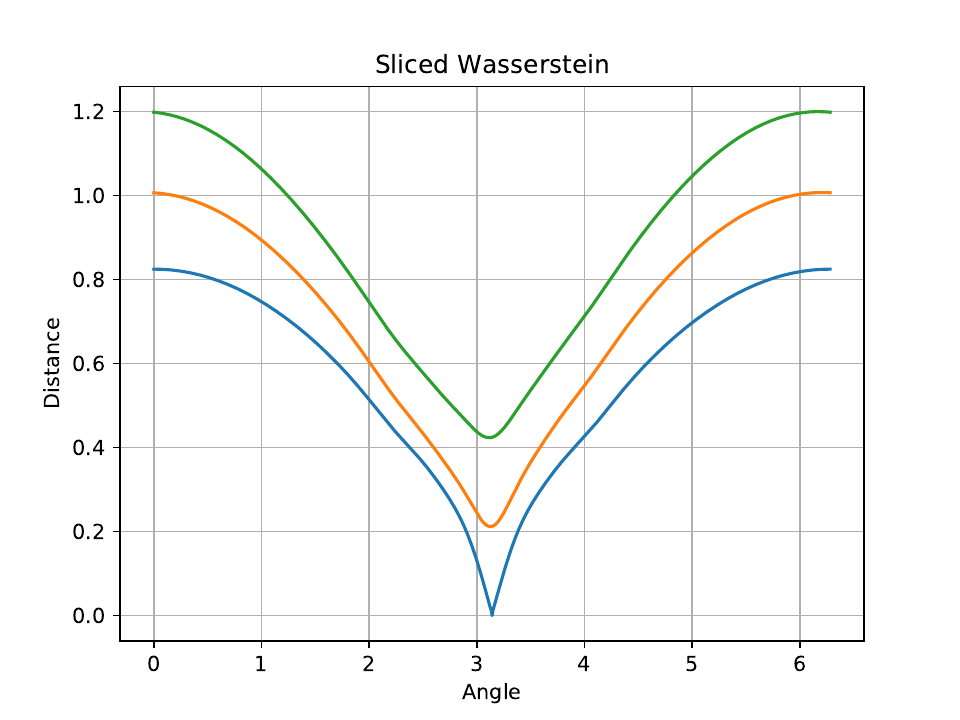}
\caption{Distances for the experiment described
in Section \ref{sec:transl_rots}.
From top to bottom, left to right: $\D_{1/2}$, $\D_{1}$, $\D_{2}$,
Euclidean, sliced Cram\'er, and sliced Wasserstein.
The blue lines are the distances between the function
and itself (rotated by $\pi$ radians so the minimum
value of $0$
occurs in the middle of the plot);
the orange lines are the distances between the function
and its translation by $0.3$;
the green lines are the distances between the function
and its translation by $0.6$.
}
\label{fig:transl_rots_distances}
\end{figure}

\begin{figure}[h]
\centering
\includegraphics[scale=.3]{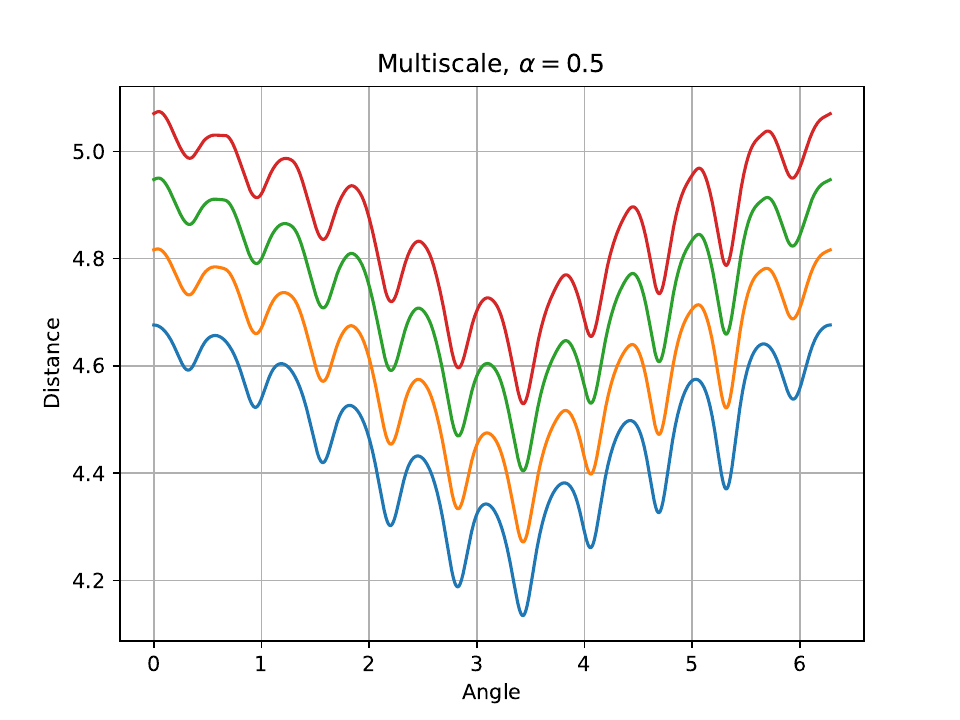}
\includegraphics[scale=.3]{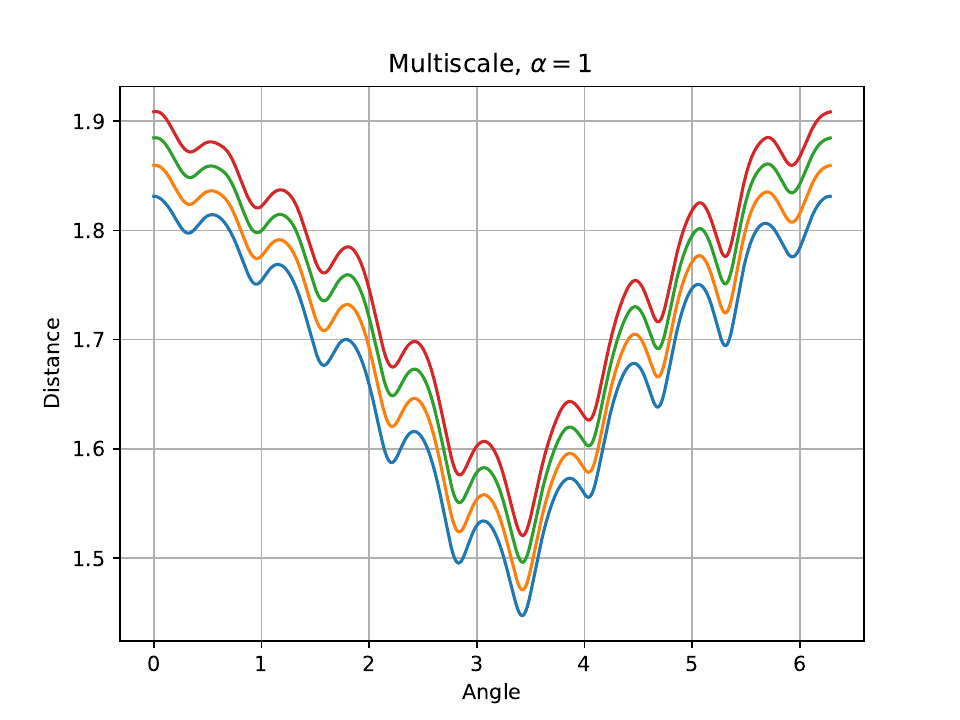}
\includegraphics[scale=.3]{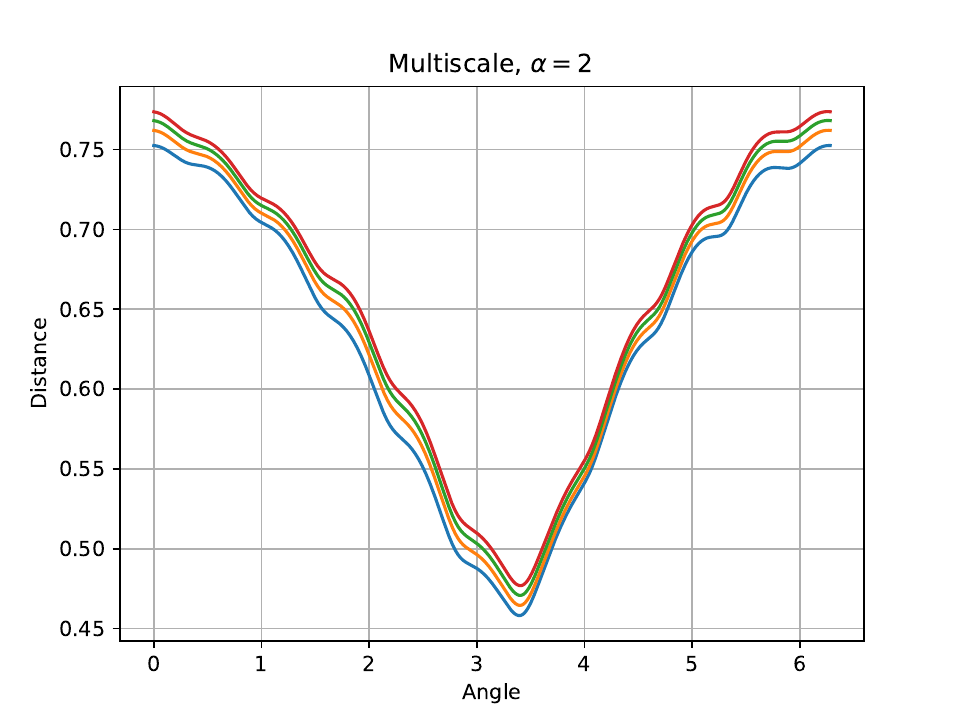}
\\
\includegraphics[scale=.3]{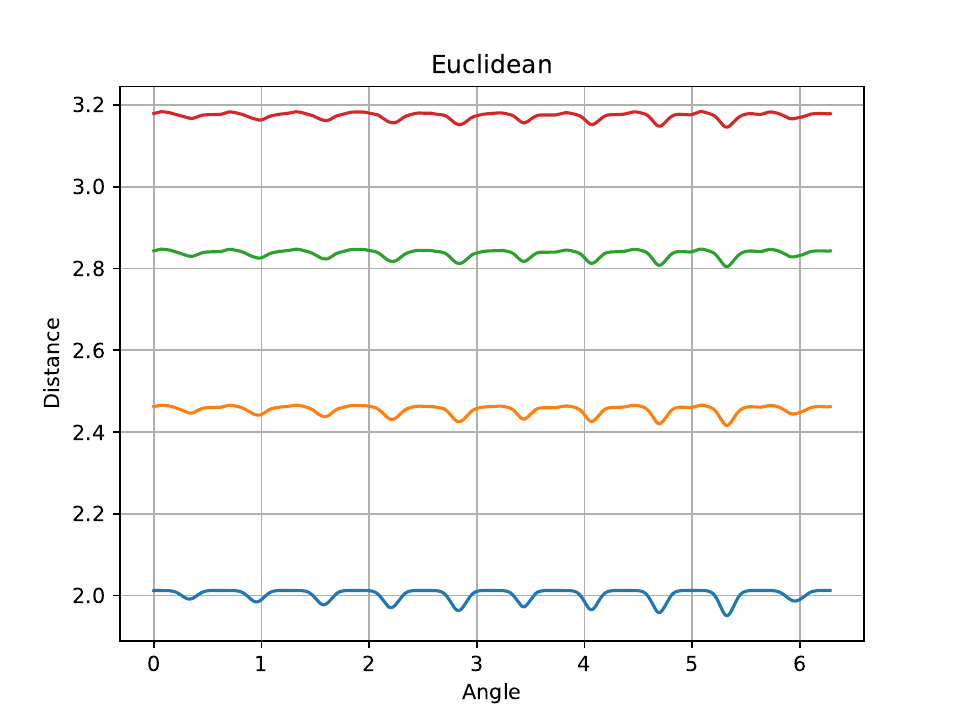}
\includegraphics[scale=.3]{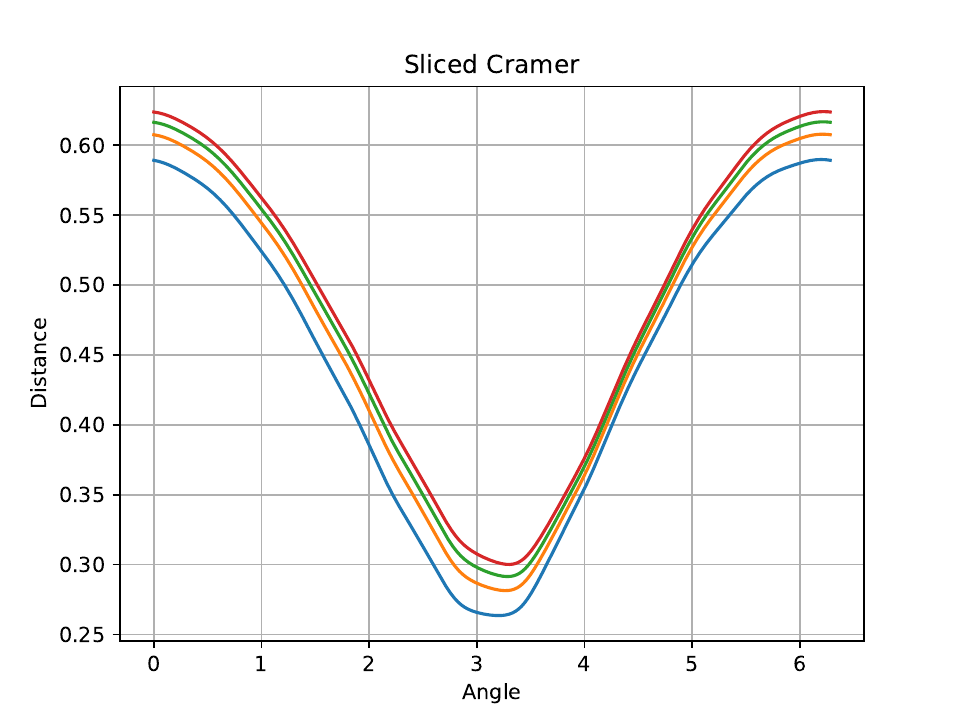}
\caption{Distances for the experiment described
in Section \ref{sec:noise_rots}.
From top to bottom, left to right: $\D_{1/2}$, $\D_{1}$, $\D_{2}$,
Euclidean, and sliced Cram\'er.
The blue lines are the distances between the function
and itself (rotated by $\pi$ radians so the minimum
value of $0$
occurs in the middle of the plot);
the orange lines are the distances between the functions
with noise of maximum variance $0.5$;
the green lines are the distances between the functions
with noise of maximum variance $1$;
the red lines are the distances between the functions
with noise of maximum variance $1.5$.
}
\label{fig:transl_rots_distances}
\end{figure}

\begin{figure}[h]
\centering
\includegraphics[scale=.35]{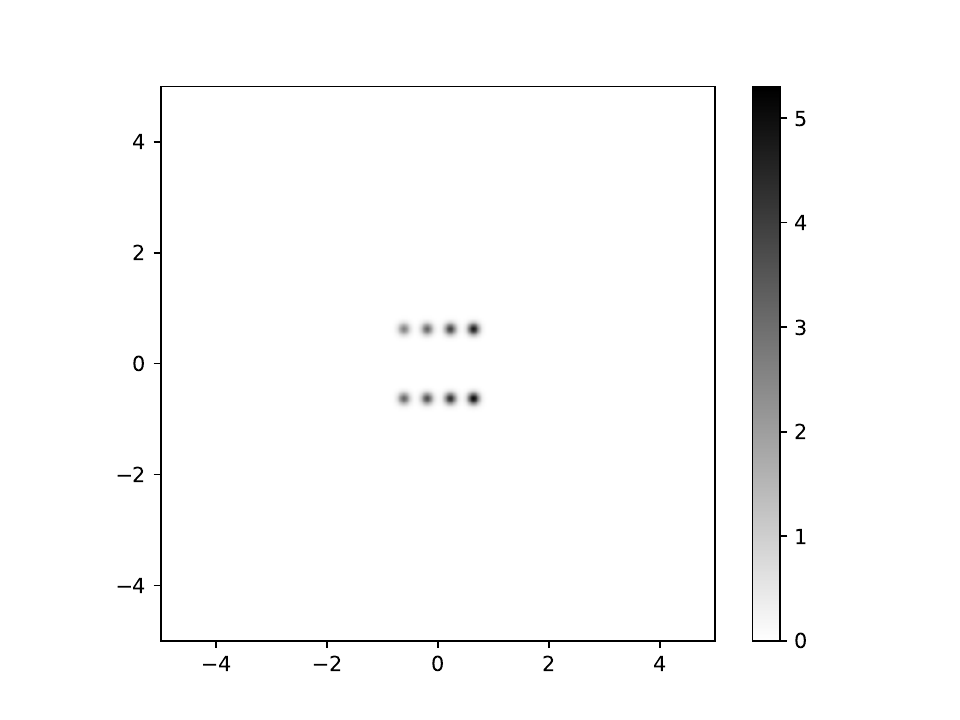}
\includegraphics[scale=.35]{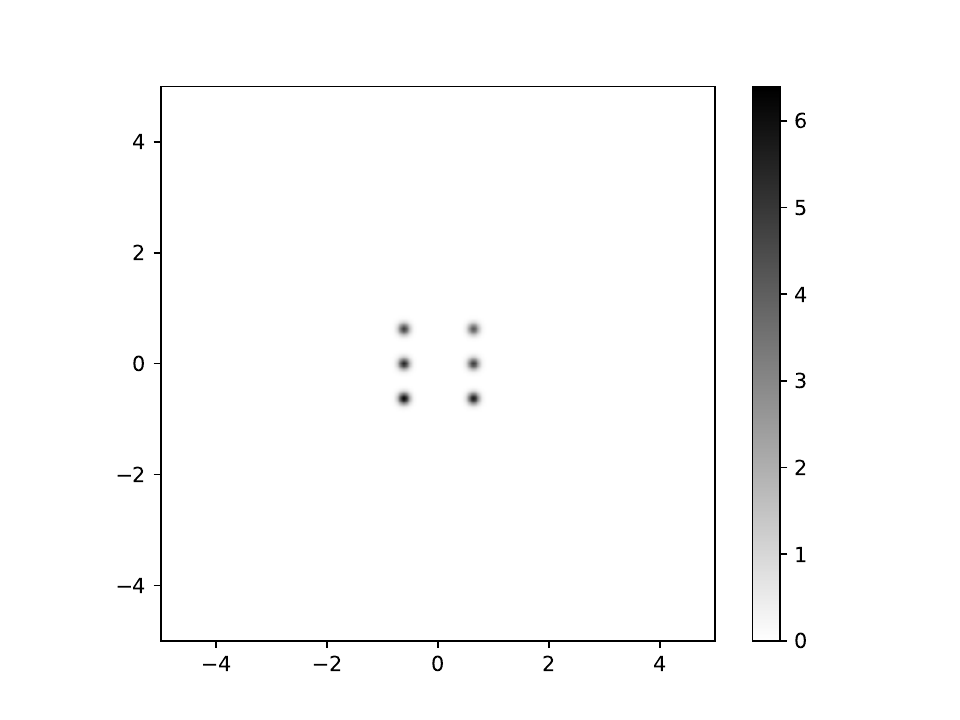}
\\
\includegraphics[scale=.35]{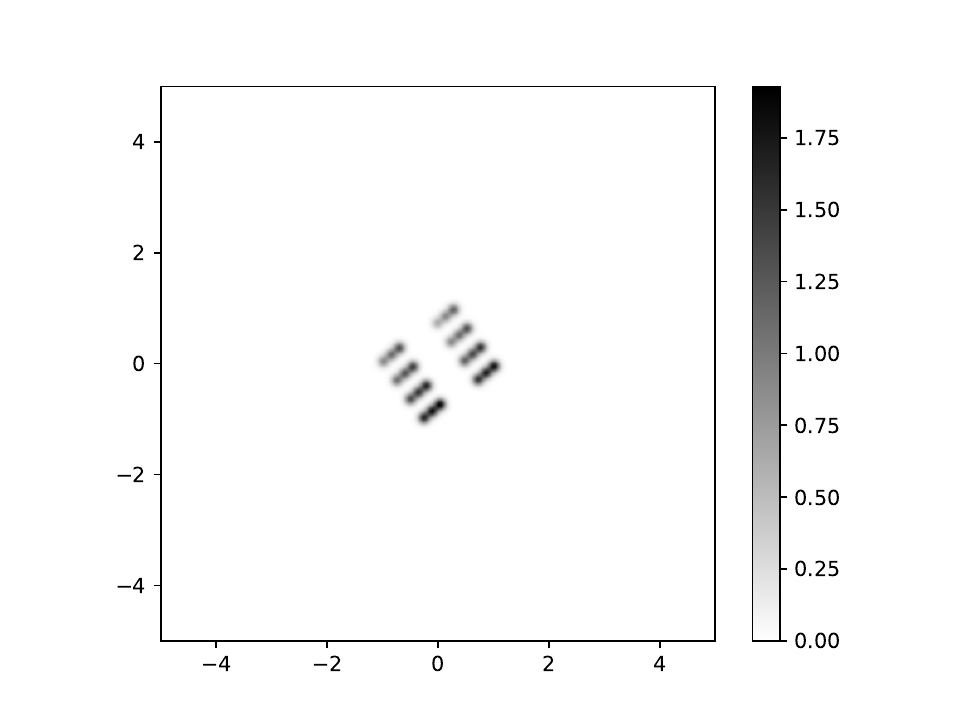}
\includegraphics[scale=.35]{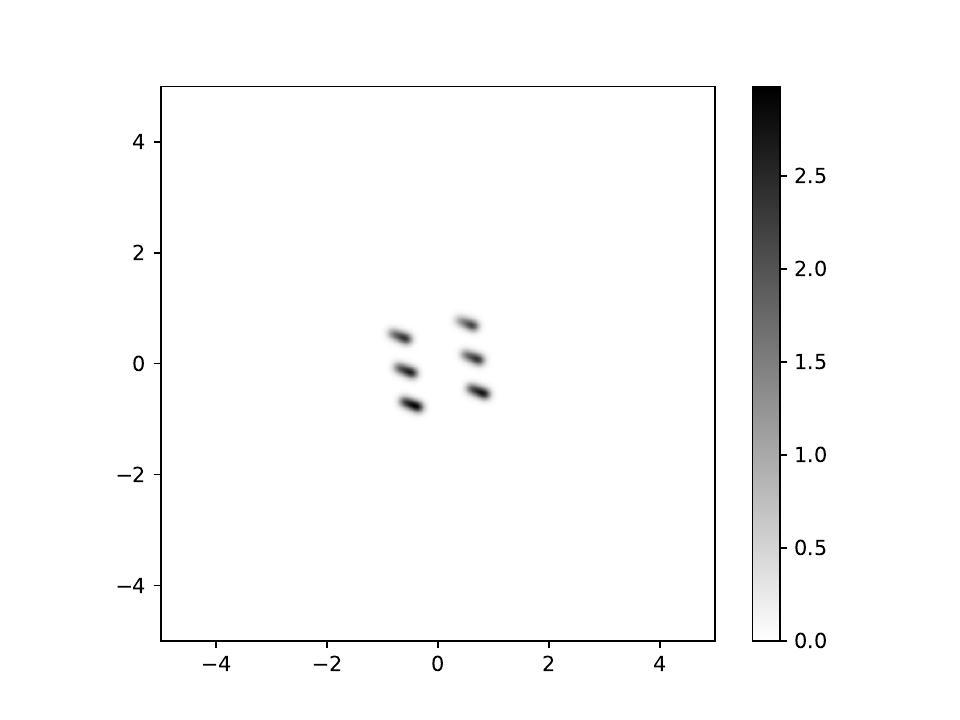}
\\
\includegraphics[scale=.35]{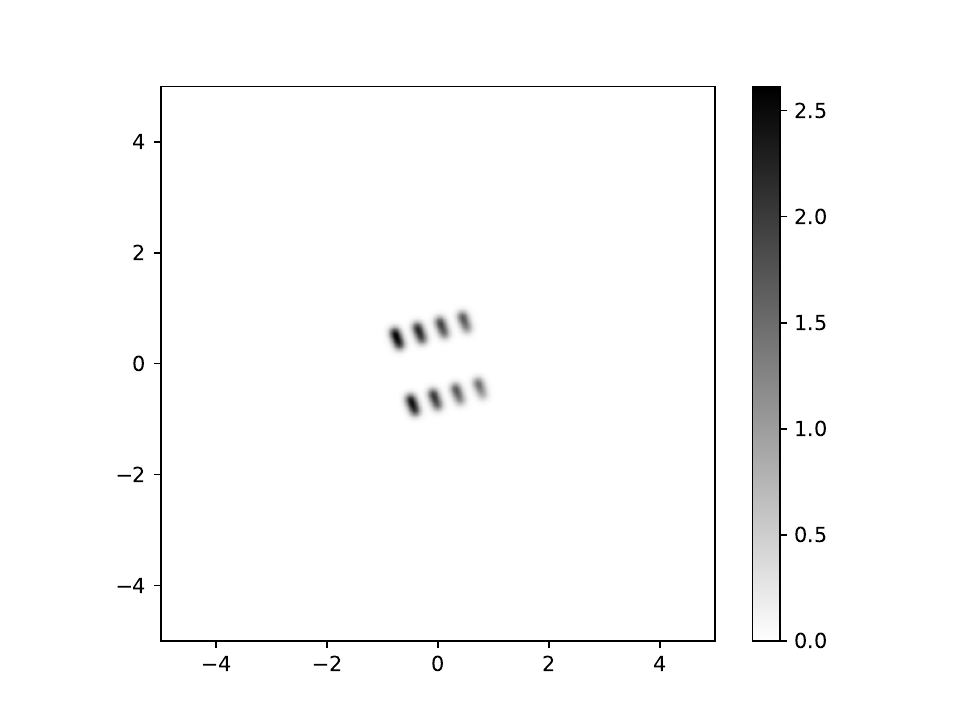}
\includegraphics[scale=.35]{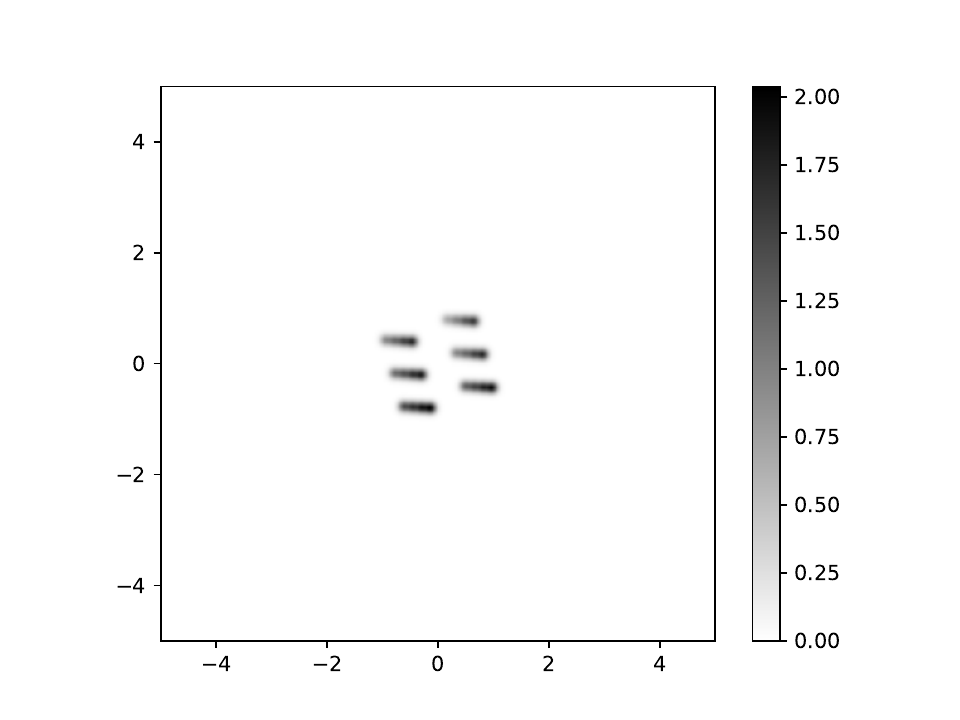}
\caption{Projections of $f$ from the experiment
described in Section \ref{sec:experiment_projections}.
Left column: the projection $g$, and two projections
of $f$ within angle $\theta$ of $u_g$.
Right column: projection $h$, and two projections
of $f$ within angle $\theta$ of $u_h$.
Here, $\theta = 3 \pi / 20$,
and the projections are shown without noise.}
\label{fig:projections}
\end{figure}

\begin{figure}[h]
\centering
\includegraphics[scale=.7]{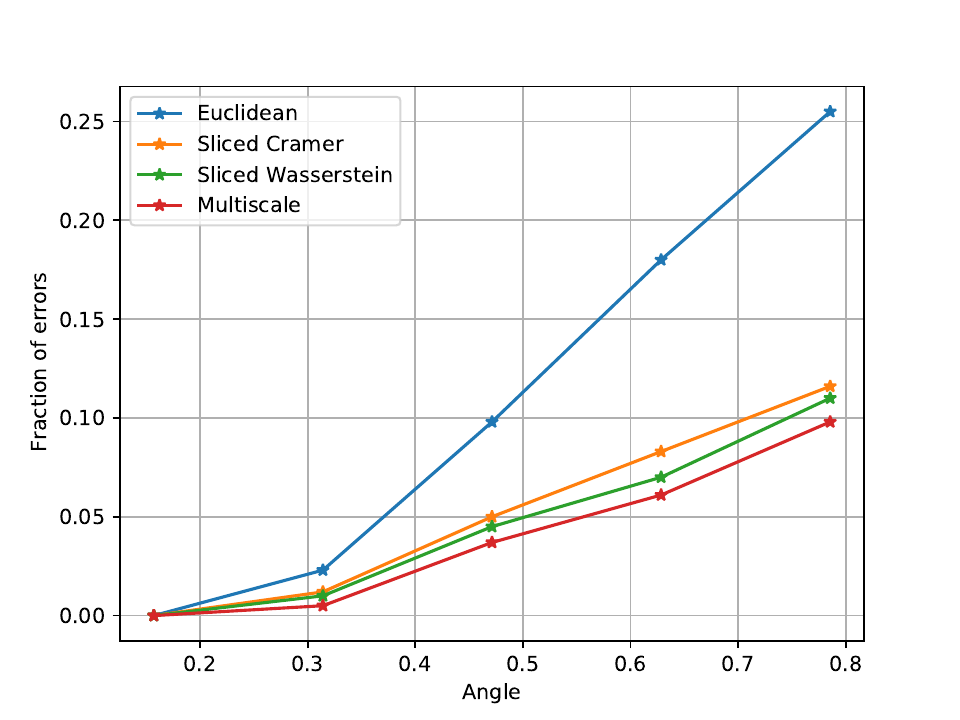}
\caption{Error plot from the experiment
described in Section \ref{sec:experiment_projections}.
}
\label{fig:errs_projs}
\end{figure}

\begin{figure}[h]
\centering
\includegraphics[scale=.7]{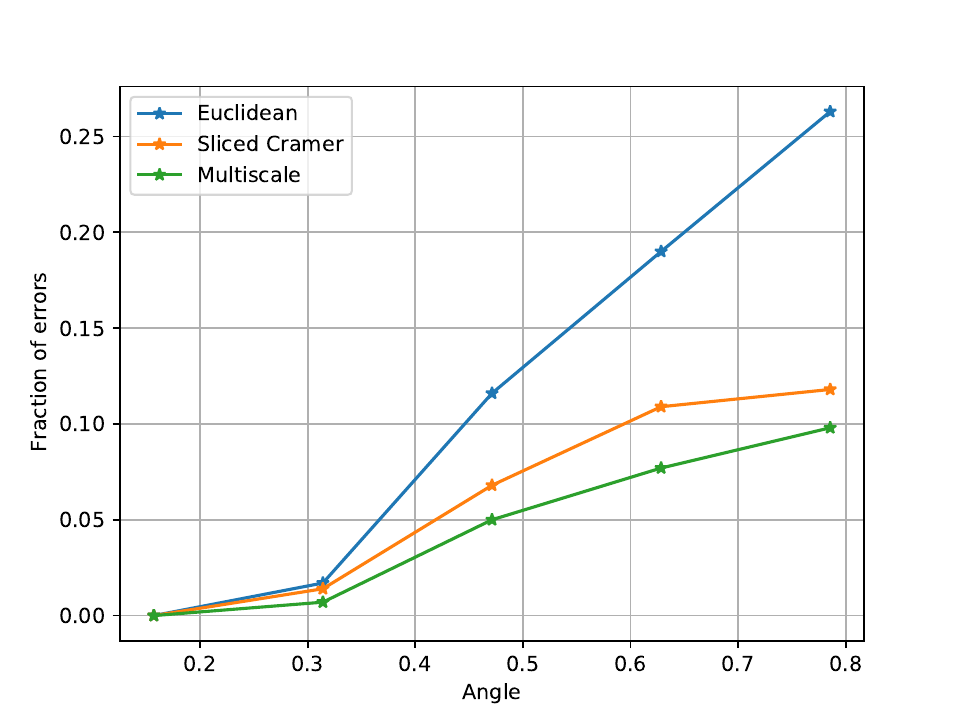}
\caption{Error plot from the experiment
described in Section \ref{sec:experiment_projections},
with noise of standard deviation $0.1$.
}
\label{fig:errs_projs_noise}
\end{figure}

\begin{figure}[h]
\centering
\includegraphics[scale=.7]{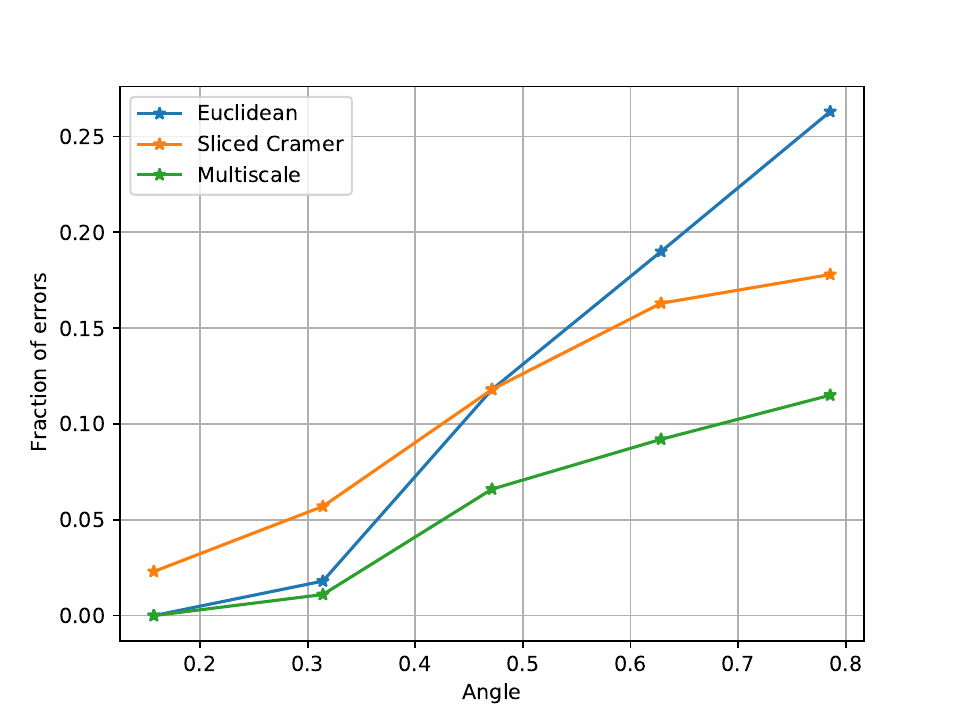}
\caption{Error plot from the experiment
described in Section \ref{sec:experiment_projections},
with noise of standard deviation $0.2$.
}
\label{fig:errs_projs_noise2}
\end{figure}

\section{Numerical experiments}

We illustrate the performance of the rotationally-invariant metrics
$\D_{\alpha,p}(f,g; \, \SO(2))$
(with $r=2$), on several numerical examples.
We compare to the rotationally-invariant Euclidean, sliced $2$-Cram\'er,
and sliced $2$-Wasserstein distances.
For simplicity, we take $p=1$, and denote the metric by $\D_\alpha$.
For our implementations of the rotationally-invariant
sliced Wasserstein and sliced Cram\'er distances,
we use similar implementations as those described
in \cite{shi2025fast}, though we make use of the
obervation from Remark \ref{rmk:finer_grid}
and perform the minimization over a finer grid
than the algorithm described in \cite{shi2025fast}.
The code for rotationally-invariant sliced Cram\'er
and sliced Wasserstein distances
may be found at \url{github.com/wleeb/SlicedCramer}.
The code for the multiscale distances $\D_{\alpha,p}$
(with $r=2$), including the rotationally-invariant
versions, may be found at \url{github.com/wleeb/Gaussian-Kernel-Metrics}.

\subsection{Translations and rotations}
\label{sec:transl_rots}

We consider the function $f$, shown in Figure \ref{fig:function_ring},
consisting of $10$ Gaussians arranged equispaced around a circle,
of varying heights. Each Gaussian is centered at distance $5/3$ from
the origin, with scale parameter $1/100$, and heights
that grow linearly starting from $1/55$ and ending at $10/55$
(the normalization is to ensure the total mass is $1$).
We evaluate distances between the function
shown and all of its rotations, using the fast algorithm
described in Section \ref{sec:discretization_rotations}. We also compute distances between
the original function and its translations to the right by $0.3$
and $0.6$, respectively.
We do this for the distances $\D_{\alpha}$, $\alpha = 0.5, 1, 2$ (with $p=1$)
as well as sliced Wasserstein, sliced Cram\'er, and Euclidean distance.
The distances are shown in Figure \ref{fig:transl_rots_distances}.
As $\alpha$ grows, the behavior of the distances $\D_\alpha$ are
smoother functions of the rotation angle, similar 
to the sliced Wasserstein and sliced Cram\'er distances;
whereas for smaller $\alpha$,
$\D_\alpha$ contains the contours of the Euclidean distance, while
still maintaining an overall envelope similar to that of the
sliced Wasserstein and sliced Cram\'er distances.

\subsection{Noise and rotations}
\label{sec:noise_rots}

We compare the rotations between the two functions shown in Figure \ref{fig:function_ring},
where white noise of varying strengths has been added to the rotated function.
In this experiment, we do not include the sliced Wasserstein distance,
as it is not designed to handle additive noise.
The functions are sampled on grids of size $1592 \times 1592$,
which is substantially greater than is needed to resolve the underlying
noiseless functions.
The variances are, respectively, $0.5$, $1.0$, and $1.5$.
As can be seen, the noise drastically inflates the Euclidean distances;
because the sample size is so large,
the effects on the other metrics are comparatively smaller,
and get smaller with larger $\alpha$, consistent with
Theorem \ref{thm:main_subgaussian}.

\subsection{Projections}
\label{sec:experiment_projections}

We consider a function $f$ in 3D, which is a convex combination
of 24 Gaussian functions arranged in a $4$-by-$3$-by-$2$ grid,
each with scale parameter $1/100$, and heights proportional
to $\sqrt{i^2 + j^2 + k^2}$, where $(i,j,k)$ is the index specifying
each Gaussian, with $i=1,2$, $j=1,2,3$, and $k=1,2,3,4$.
(The heights are normalized so the total mass is $1$, making $f$
a probability density.)
The function $f$ is centered at the origin, with the centers
of the Gaussians
equally spaced along each dimension, ranging from $-5/8$ to $5/8$.

We consider two 2D projections of $f$, whose axes of projection
have angle $\pi/2$; these two functions, which we call $g$ and $h$,
are shown in the top row of Figure \ref{fig:projections}. Denote by $u_g$ the axis
along which $f$ is projected to generate $g$, and similarly denote $u_h$.

In the experiment, for a given angle $\theta$,
we randomly generate two sets of $500$ projections of $f$,
which we denote $S_\theta(g)$ and $S_\theta(h)$,
to which we add independent Gaussian noise of variance $\sigma^2$.
Some example projections without
noise ($\sigma=0$) are shown in the bottom two rows of Figure \ref{fig:projections}.
All the projections in $S_\theta(g)$ (respectively $S_\theta(h)$) are taken
at projection axes that lie within $\theta$ of $u_g$
(respectively $u_h$).

For each projection in $S_\theta(g)$
and $S_\theta(h)$, we compute its distances to both $g$
and $h$, using the rotationally-invariant versions of the $\D_{2}$, Euclidean,
sliced Cram\'er, and, in the $\sigma=0$ case, sliced Wasserstein metrics,
and for each distance assign that projection to its nearest center (either $g$
or $h$).
For the metric $\D_2$, we take $\tau_0 = 25/16$.

In Figure \ref{fig:errs_projs}, we plot the fraction of incorrectly classified
noiseless ($\sigma=0$) projections, as the angle $\theta$ grows,
for all four distances.
In Figures \ref{fig:errs_projs_noise}
and \ref{fig:errs_projs_noise2}
we plot the fraction of incorrectly classified projections,
for $\sigma=0.1$ and $\sigma=0.2$, respectively,
as the angle $\theta$ grows;
in these latter two plots, we exclude the sliced Wasserstein
metrics, which are not designed to handle additive noise.
As one would expect, all errors grow with $\theta$,
since the projections differ more as the angle between
them changes.

Interestingly, in all cases the metric $\D_{2}$
maintains the smallest errors. In
the noiseless case, the errors are close to the errors
of sliced Cram\'er and sliced Wasserstein, all three of
which are considerably smaller than the errors of Euclidean distance.
With non-zero noise, the metric $\D_2$ performs
substantially better than both the Euclidean
and sliced Cram\'er distances.
This is suggestive, though by no means conclusive,
that the multiscale metric may perform well
on the class averaging problem in cryo-EM
\cite{park2011stochastic, hadani2011representation,
bhamre2017mahalanobis, rao2020wasserstein, zhao2014rotationally};
however, additional experimentation is needed
that goes beyond the scope of the current paper
and will be pursued in future work.

\section{Conclusion}

We have shown that the metrics $\D_{\alpha,p,r}$
are stable to geometric deformations,
and that their discretizations are stable under additive noise.
We have also introduced rotationally-invariant versions
of $\D_{\alpha,p,2}$, and we have shown that they may be computed
efficiently from finite samples. Our numerical
results indicate favorable performance
compared to the sliced Wasserstein, sliced Cram\'er,
and Euclidean distances.

There are several open questions that we intend to address
in future work. Perhaps most straightforward
is minimizing the distance over all rigid motions,
not just over $\SO(2)$.
While minimizing over translations alone is fairly
straightforward (more so than over rotations as we have done
in the present work, as it does not require
interpolating to a polar grid), simultaneous minimization over
all rigid motions is a more challenging task,
though we are optimistic that the methods
in \cite{rangan2020factorization, rangan2023radial}
may be applicable.
Second is the question of how to select the parameters
$\alpha$ and $p$;
in particular, larger $\alpha$ puts more weight
on the low-frequency components of the inputs,
which makes the metric more stable but may erase
critical details in the higher frequencies.
On the other hand, larger $p$ approaches the $\infty$
norm across scales, which selects a single
scale to discriminate the two images,
whereas values of $p$ close to $1$ incorporate all scales simultaneously.
We also intend to
extend this work to comparing volumes;
Wasserstein-type distances have been proposed
for this task in the context of cryo-EM \cite{singer2024alignment, zelesko2020earthmover},
and so it is reasonable to generalize the multiscale metrics
studied herein
to apply the same class of problems.
Finally, Theorem \ref{thm:rotations_projections} and
the experiments in Section \ref{sec:experiment_projections}
suggest that in certain cases the rotationally-invariant metrics $\D_{\alpha,p}$
may be good choices for clustering noisy tomographic
images, a task that occurs in the class averaging problem
in cryo-EM \cite{park2011stochastic, hadani2011representation,
bhamre2017mahalanobis, rao2020wasserstein, zhao2014rotationally};
in future work we intend to design and test an end-to-end
clustering algorithm for this task.

\subsection*{Acknowledgements}

I acknowledge funding from NSF CAREER award DMS-2238821
and the McKnight Foundation.

\clearpage
\bibliographystyle{plain}
\bibliography{refs}

\end{document}